\documentclass[12pt,fleqn]{amsart}

\usepackage{macros}
\begin{document}

\title{Proper maps of annuli}
\author[Abdullah Al Helal]{Abdullah Al Helal\,\orcidlink{0000-0002-4609-4511}}
\address{Department of Mathematics, Oklahoma State University, Stillwater, OK 74078-5061}
\email{ahelal@okstate.edu}

\author[Ji{\v r}{\' i} Lebl]{Ji{\v r}{\' i} Lebl\,\orcidlink{0000-0002-9320-0823}}
\address{Department of Mathematics, Oklahoma State University, Stillwater, OK 74078-5061}
\email{lebl@okstate.edu}
\thanks{The second author was in part supported by Simons Foundation collaboration grant 710294.}

\author[Achinta Kumar Nandi]{Achinta Kumar Nandi\,\orcidlink{0009-0006-4129-3649}}
\address{Department of Mathematics, University of California San Diego, La Jolla, CA 92093-0112}
\email{anandi@ucsd.edu}

\date{November 9, 2026}


\begin{abstract}
We study proper holomorphic maps of annuli in complex Euclidean spaces,
that is, domains with $U(n)$ as the automorphism group.
By the Hartogs phenomenon and a result of Forstneri\v{c}, such maps are always
rational and extend to proper maps of balls.
We first prove that a proper map of annuli from $n$ dimensions to $N$ dimensions
where $N < \binom{n+1}{2}$ is always an affine embedding.  This inequality
is sharp as the homogeneous map of degree 2 satisfies $N=\binom{n+1}{2}$.
Next we find a necessary and sufficient condition for a map to be
homogeneous: A proper map of annuli is homogeneous if and only if
its general hyperplane rank, the affine dimension of the image of
a general hyperplane, is exactly $N-1$.
As a corollary, we obtain a classification of homogeneous
proper maps of balls.
A homogeneous proper ball map takes all spheres centered at the origin
to spheres centered at the origin.
We show that if a proper ball map has
general hyperplane rank $N-1$ and takes one sphere centered
at the origin to a sphere centered at the origin, then it is homogeneous.
Another corollary of this result is a complete
classification of proper maps of annuli from dimension 2 to dimension 3.
Finally, we give a complete normal form of rational proper maps of annuli
of degree 2.
\end{abstract}

\maketitle

%%%%%%%%%%%%%%%%%%%%%%%%%%%%%%%%%%%%%%%%%%%%%%%%%%%%%%%%%%%%%%%%%%%%%%%%%%%%%

% Allow breaking of displayed environments like align
% across pages to avoid really weird stretched pages
\allowdisplaybreaks

\section{Introduction}

One of the fundamental questions in complex analysis is to understand proper
holomorphic maps $f \colon U \to V$ between two domains $U \subset \C^n$
and $V \subset \C^N$.  In general, few if any such maps exist unless $U$
and $V$ are related.  If $U$ and $V$ are the same type of domains
and have many symmetries, then many such maps exist in general.
The domain with the ``largest'' automorphism group is the unit ball
$\B_n \subset \C^n$ and the space of proper maps of balls $f \colon \B_n \to
\B_N$ has a rich structure and has been extensively studied.
We will focus on proper maps of annuli, as annuli are domains with a large
automorphism group.
Write
$\A_{n,r} = \B_n \setminus \overline{r\B_n} = \{ z \in \C^n : r < \norm{z}
< 1 \}$ for an annulus, where
we will always assume that $0 < r < 1$.
The automorphism group of $\A_{n,r}$ is the unitary group
as long as $n > 1$.  When $n=1$ there are also the automorphisms $e^{i\theta}
\frac{r}{z}$.
Notice that unlike the ball, the annulus is non-pseudoconvex; its boundary
consists of a
strongly pseudoconvex part and a strongly pseudoconcave part.

It is not difficult to show that if $n > 1$,
then proper holomorphic maps $f \colon \A_{n,r}
\to \A_{N,R}$ extend to proper maps of balls that also take the
$r$-sphere to the $R$-sphere.
Their study therefore involves proper maps of balls.
Our goal is to prove results for annuli analogous to the known results for
balls and to contrast where these results fundamentally differ.
We therefore very briefly discuss some of the known results for balls.

Consider then a proper holomorphic map $f \colon \B_n \to \B_N$.
If $n=N=1$, then by a classical theorem of Fatou~\cite{fatou-1923-fonctions}, 
$f$ is a finite Blaschke
product and hence rational.  If $n = N > 1$, then by a theorem of
Alexander~\cite{alexander-1977-proper}, $f$ is an automorphism of the unit ball.
%No maps exist if $N < n$ for elementary reasons.
Dor~\cite{dor-1990-proper} and L\o w~\cite{low-1985-embeddings} show that if $N > n > 1$, 
there exist nonrational
maps, even in the case of codimension 1, that is, $N= n+1$.
If $n > 1$ and the map is sufficiently smooth up to the boundary, then by
a theorem of Forstneri\v{c}~\cite{forstneric-1989-extending}, 
all such $f$ are rational with the degree 
bounded by an integer depending on $n$ and $N$.
Faran~\cite{faran-1982-maps} showed that
proper maps of $\B_2$ to $\B_3$ that are
$C^3$ up to the boundary are \emph{spherically equivalent} (up to
composition with automorphisms of the balls) to one of four polynomial maps.
For $n \geq 3$,
by a series of papers of Webster~\cite{webster-1979-mapping},
Faran~\cite{faran-1986-linearity}, and Huang~\cite{huang-1999-linearity},
every proper map from $\B_n$ to $\B_N$ for $N \leq 2n-2$ that is
sufficiently smooth up to the boundary is spherically equivalent 
to the linear embedding.
For $N=2n-1$, the Whitney map $(z',z_n) \mapsto (z',z_n \otimes z)$
is not spherically equivalent to the linear embedding.
Interestingly, this map has no analogue for annuli.
In general, see the books by 
D'Angelo~\cites{dangelo-1993-several,dangelo-2019-hermitian,dangelo-2021-rational}
for a more thorough treatment of proper maps of balls.

Let us move to maps of annuli.
When $n=1$, the same classical argument to compute the automorphism group of $\A_{1,r}$
extends to show that a proper holomorphic map
$f \colon \A_{1,r} \to \A_{1,R}$ exists only
if $R=r^d$ for some $d \in \N$ in which case $f(z)=e^{i\theta} z^d$ or
$f(z)=r^d e^{i\theta} z^{-d}$.
As we said before, if $n > 1$, annulus maps extend to ball maps.
By Forstneri\v{c}, the maps must be rational.
See \cref{theorem:annulimapsrational}
(proved in \cite{helal-2025-proper}) for more details.
So when $n > 1$, holomorphic proper maps of annuli are in one-to-one
correspondence with rational proper maps of balls that take a sphere
centered at zero to another sphere centered at zero.

A natural map that is a proper map for both balls and
annuli is the homogeneous map $z^{\otimes d}$
(the $d$th tensor power).
That is because if $\norm{z}^2$ is constant, then
\begin{equation}
\norm{z^{\otimes d}}^2 = 
\norm{z}^{2d}
\end{equation}
is constant.
This map can be symmetrized (as $z^{\otimes d}$ contains some duplicate
entries) via a unitary map to what we will call $H_d$ direct sum a
number of zero components.  In particular,
$\norm{H_d(z)}^2 = \norm{z^{\otimes d}}^2$.
Rudin~\cite{rudin-1984-homogeneous} proved that every homogeneous
proper ball map is, up to a unitary, $H_d \oplus 0$.  
See D'Angelo~\cites{dangelo-1988-polynomial,dangelo-2003-proper,dangelo-1993-several} for a much simpler
proof of this fact and many more details and results about the homogeneous map.

Hence, $H_d$ is a map that takes any sphere of radius $r$ centered at the
origin to a sphere of radius $r^d$ centered at the origin, and hence
takes
$\A_{n,r}$ (for any $0 < r < 1$) properly to $\A_{N,R}$, where $N = \binom{n+d-1}{d}$ is
the number of components of $H_d$ and $R=r^d$.
We can now construct new maps by what we will call \emph{juxtaposition}.
Given two proper annulus maps with the same source,
$F \colon \A_{n,r} \to \A_{N_1,R_1}$
and
$G \colon \A_{n,r} \to \A_{N_2,R_2}$
and $t \in (0,1)$,
then $\sqrt{1-t}\, F \oplus \sqrt{t}\, G \colon \A_{n,r} \to \A_{N,R}$ is a proper map where
$N = N_1 + N_2$ and $R^2 = (1-t) R_1^2+t R_2^2$.
Based on the one-dimensional result, noting that $z^{\otimes d}$ is the
correct multidimensional analogue of $z^d$, one could (incorrectly)
conjecture that all such proper maps are juxtapositions of homogeneous
maps.
However, other maps do exist, see
\cref{sec:seconddeg}
and more generally \cite{helal-2025-proper}.

Our goal is to classify proper holomorphic
maps $f \colon \A_{n,r} \to \A_{N,R}$.
We say
$f \colon \A_{n,r} \to \A_{N,R}$
and
$g \colon \A_{n,r} \to \A_{N,R}$
are \emph{unitarily equivalent} (or simply \emph{equivalent}) if there
exist unitary matrices $V \in U(n)$ and $W \in U(N)$ so that
\begin{equation}
f = WgV .
\end{equation}
As we noted, $U(n)$ is the holomorphic automorphism group of $\A_{n,r}$.
If $n=N > 1$, then a proper map $f \colon \A_{n,r} \to \A_{n,R}$
exists only if $r=R$ and $f$ is unitarily equivalent to the identity $z$,
see \cref{corollary:equidimdeg1}.
If $N > n$, then an affine embedding
exists whenever $r \leq R < 1$, but is not necessarily $z \oplus 0$
due to the different automorphism group.
In particular,
if $f \colon \A_{n,r} \to \A_{N,R}$ is a proper rational map of degree 1,
then $f$ is unitarily equivalent to the affine embedding
\begin{equation} \label{eq:linearembed}
z \mapsto \sqrt{\frac{1-R^2}{1-r^2}} \, z \oplus
\sqrt{\frac{R^2-r^2}{1-r^2}} \oplus 0 ,
\end{equation}
where the last $0$ has $N-n-1$ terms.
That is, we shift the linear
embedding by adding the proper constant in the $(n+1)$th component.
Allowing for constants in juxtaposition,
this map is a juxtaposition of the identity $z$
and the constant map $1$.
See \cref{prop:linearembeddings}.

As we noted above, for sufficiently smooth proper maps between balls,
$f \colon \B_n \to \B_N$ no maps other than the linear embedding
exist for $N < 2n-1$, and then when $N=2n-1$, the Whitney map exists
and is not spherically equivalent to a linear embedding.
Curiously, for annulus maps, this gap is much longer and goes until
$N= \binom{n+1}{2}$.

\begin{theorem}
\label[theorem]{thm:gapthm1}
Suppose $f \colon \A_{n,r} \to \A_{N,R}$ is a proper holomorphic map
and $2 \leq n < N < \binom{n+1}{2}$.
Then $r \leq R < 1$ and $f$ is unitarily equivalent to an affine embedding
\eqref{eq:linearembed}.
The bound on $N$ is sharp as the homogeneous map $H_2$
takes the annulus $\A_{n,r}$ properly to $\A_{\binom{n+1}{2},r^2}$.
\end{theorem}

As homogeneous maps play a significant role in the study of proper maps of
annuli, we next find a property that determines a map to be homogeneous.
We define an invariant $k_f$ that has appeared
in many guises in the study of ball maps.
Let $f \colon U \subset \C^n \to \C^N$ be any map.  If $X \subset \C^N$,
write $\aff(X)$ for the affine hull of $X$, that is, the smallest
affine subspace of $\C^N$ containing $X$.
Define
\begin{equation}
k_f \overset{\text{def}}{=}
\max
\left\{
\dim \aff\bigl(f(H \cap U)\bigr) :
H \subset \C^n \text{ is an affine hyperplane and }
H \cap U \not= \emptyset
\right\} .
\end{equation}
We call $k_f$ the \emph{general hyperplane rank}.  One could think of $k_f$
as a certain measure of complexity.
For a holomorphic map $f$, $k_f$ is attained on a generic hyperplane,
and it is in general quite easy to compute, especially for rational maps.
%As a side note, this property makes $k_f$ extremely
%easy to compute probabilistically;
%pick a random hyperplane in $\C^n$,
%pick at least $n$ points on that hyperplane, and compute the dimension
%of their affine span.
While for an arbitrary holomorphic map one generally expects $k_f = N$,
via a Segre-variety argument, a rational proper map of balls (or annuli)
has $k_f \leq N-1$.
This fact has been used many times in the study of ball maps.
The integer $k_f$ is also related to the $X$-variety.
See \cref{sec:kf} for more on this connection and the history.
A related concept is the embedding dimension $N_f$,
that is, the dimension of the smallest affine subspace $S$ containing
$f(\A_{n,r})$ (or $f(\B_n)$).
In other words, $N_f = \dim \aff f(\A_{n,r})$.
In either case, restricting the codomain to $S$
obtains a proper map between annuli (or balls) with target
dimension $N_f$.  Note that it is not always true
that $k_f = N_f-1$.
If $f$ is a juxtaposition of two rational maps
which do not differ by a juxtaposition with a constant map followed by a target unitary,
then $k_f \leq N_f-2$, see
\cref{proposition:juxtaposekf}.
What we prove is that for proper maps of annuli, $k_f = N_f-1$ is
precisely the condition that characterizes homogeneous maps $H_d$.

\begin{theorem} \label[theorem]{thm:kfN1impliesHd}
Suppose $n \geq 2$ and $f \colon \A_{n,r} \to \A_{N,R}$ is a
proper holomorphic map.  Then
$k_f = N-1$
if and only if
$N=\binom{n + d - 1}{d}$, $R=r^d$, and
\begin{equation}
f = U H_d ,
\end{equation}
for some $d \in \N$ and a unitary $U \in U(N)$.
\end{theorem}

One direction of this theorem is obvious.
The remarkable direction is that $k_f = N -1$
implies that $f$ is homogeneous.
The proof relies on a result from
\cite{helal-2025-proper}, where 
the authors showed that a rational map that takes
infinitely many spheres
centered at the origin to spheres centered at the origin is a juxtaposition
of homogeneous maps.
To state the theorem in terms of $N_f$,
we ought to say that
$k_f = N_f-1$ if and only if $f$ is an affine embedding of a
homogeneous map.
Note that $k_f = N-1$ implies $N=N_f$ for any proper rational map
of annuli (or balls).

Let us state a version of the forward direction of this result for
proper maps of balls.
Clearly $H_d$ is a proper ball map that takes infinitely many spheres centered
at the origin to spheres centered at the origin.
With the assumption that $k_f = N-1$, to show that $f$
is homogeneous, it is sufficient
to check that a map takes just
one inner sphere centered at the origin to a sphere centered at the origin.

\begin{theorem} \label[theorem]{thm:kfN1impliesHdv2}
Suppose $n \geq 2$ and $f \colon \B_n \to \B_N$ is a proper
holomorphic
map such that $k_f = N-1$
and $f(rS^{2n-1}) \subset RS^{2N-1}$ for some $r,R \in (0,1)$.
Then there exists a unitary map $U$ and an integer $d$ such that
$f = U H_d$.
\end{theorem}

To prove the theorem, we note that we have a rational map taking two
distinct spheres to two distinct spheres.  The condition
$k_f=N-1$ is then used to show that $f$
must in fact take infinitely many spheres to spheres.

As a corollary to \cref{thm:kfN1impliesHd} and to the result
of Faran classifying proper maps from $\B_2$ to $\B_3$,
we completely classify proper holomorphic maps from $\A_{2,r}$ to
$\A_{3,R}$.

\begin{theorem} \label[theorem]{theorem:farananalogue}
Suppose $f \colon \A_{2,r} \to \A_{3,R}$ is a proper holomorphic map.
Then $f$ is unitarily equivalent to exactly one of the following two maps:
\begin{enumerate}
\item
The affine embedding
\begin{equation}
(z_1,z_2) \mapsto
\left(
\sqrt{\frac{1-R^2}{1-r^2}} \, z_1 ,
\sqrt{\frac{1-R^2}{1-r^2}} \, z_2 ,
\sqrt{\frac{R^2-r^2}{1-r^2}} \right),
\end{equation}
where $R \geq r$.  Note that $N_f = 2$ and $k_f = 1$.
\item
The homogeneous map $H_2$
\begin{equation}
(z_1,z_2) \mapsto
\left(
z_1^2 ,
\sqrt{2} \, z_1 z_2 ,
z_2^2 \right),
\end{equation}
where $R = r^2$.  Note that $N_f = 3$ and $k_f = 2$.
\end{enumerate}
\end{theorem}

It is also interesting to classify maps by degree.  Degree one maps are
unitaries possibly composed with an affine embedding if the target
dimension is larger than the embedding dimension.  We completely
classify degree two maps in \cref{sec:seconddeg}.  There exist other degree
two maps than the homogeneous map $H_2$ or a juxtaposition of $H_2$ and $H_1$ and/or $H_0$.
For proper ball maps,
all degree two maps are spherically equivalent to monomial maps
by the work of the second author~\cite{lebl-2011-normal}.
For annuli maps where the automorphism group is the unitary,
the conclusion is more elaborate.
For every $r$ and $R$ with
$r^2 < R < 1$, there exists a one-parameter family of rational proper maps from
$\A_{n,r}$ to $\A_{N,R}$, where $N_f = N = \binom{n+1}{2}+n$ and these maps
are not equivalent to a polynomial map.

Let us give the organization of this paper.  In \cref{sec:basic,sec:kf}, we state
and prove the basic results for maps of annuli
and the general hyperplane rank $k_f$.
In \cref{sec:gap},
we prove the first gap theorem, that is, \cref{thm:gapthm1}.
In \cref{sec:homog}, we prove the classification of homogeneous maps
in terms of $k_f$, that is, \cref{thm:kfN1impliesHd}.  We also
prove the classification of annuli maps in dimension 2 and 3, that is,
\cref{theorem:farananalogue}.
In \cref{sec:seconddeg}, we give a complete classification of second degree rational annuli maps.
Finally, in \cref{sec:remarks}, we give some further remarks and open
questions.

The authors would like to acknowledge John D'Angelo for many insightful
comments on this work, and the second author in particular would like to
acknowledge D'Angelo for introducing him to the problems surrounding proper
maps of balls and for the many years of extremely fruitful discussions on
related topics.  The authors would also like to acknowledge Dusty Grundmeier
for useful suggestions on this work.
The authors would also like to thank the anonymous referee for a careful
reading of the manuscript and for very helpful suggestions and corrections.

% ==============================================================

\section{Basic results on proper maps of annuli}\label[section]{sec:basic}

Let us start with proving some of the basic statements we made
in the introduction.  First, we have the classical result about proper maps
of annuli in one dimension.

\begin{proposition}
Suppose $f \colon \A_{1,r} \to \A_{1,R}$ is a proper holomorphic
map.  Then there is some $d \in \N$ and $\theta \in \R$ so
that $R=r^d$ and $f(z) = e^{i\theta} z^d$ or $f(z) = r^d e^{i\theta} z^{-d}$.
\end{proposition}

\begin{proof}[Sketch of proof]
Let us sketch the proof as it may not be easy to find, but
it is almost identical to classifying biholomorphisms of annuli.
After a possible inversion and scaling,
we assume that $\abs{f(z)}$ goes to 1 as we approach the outer circle
and it goes to $R$ as we approach the inner circle.
The map $z \mapsto \frac{\log \abs{f(z)}}{\log R}$ is
continuous on $\overline{\A_{1,r}}$, harmonic inside and is 0
on the outer circle and 1 on the inner circle.  Hence it
is equal to $\frac{\log \abs{z}}{\log r}$,
and we get $\abs{f(z)} = \abs{z}^{\frac{\log R}{\log r}}$.
So $f(z) = e^{i\theta} z^d$ and $R=r^d$ for some $d \in \N$.
\end{proof}

In \cite{helal-2025-proper}, the authors have proved, using the theorem of Forstneri\v{c}
\cite{forstneric-1989-extending} as a key ingredient, that all annulus maps are rational when $n
\geq 2$.  For completeness we reproduce it here.

\begin{theorem}[\cite{helal-2025-proper}] \label[theorem]{theorem:annulimapsrational}
Suppose $n \geq 2$ and $f \colon \A_{n,r} \to \A_{N,R}$ is a proper holomorphic map.
Then $f$ is a rational map of degree bounded
by some expression $D(n,N)$.  Moreover, $f$ extends to a holomorphic map on
$\B_n$ and gives a proper map of $\B_n$ to $\B_N$ and also a proper map of
$r\B_n$ to $R\B_N$.
\end{theorem}

Unlike in the ball case, no hypothesis on the boundary regularity
is needed due to the pseudoconcavity of the inner boundary.
The best expression for $D(n,N)$ is not known, though D'Angelo has
conjectured~\cite{dangelo-2003-sharp} that 
for ball maps, it is $\frac{N-1}{n-1}$ if $n \geq 3$ and $2N-3$ if $n=2$.
For annulus maps, we know that the $D$ must be smaller for some dimensions.  For example, by our result for $n=2$ and $N=3$, we see that $D(2,3)=2$ in the annulus case, while Faran shows that in the ball case the degree bound is $3$.

\begin{proof}[Sketch of proof]
As $n \geq 2$, the map $f$ extends holomorphically through the inner ball
via Hartogs phenomenon and the image is contained in $\B_N$ by the maximum
principle.  It must take the inner sphere to the inner sphere
as it takes the pseudoconcave part of the boundary of the source
to the pseudoconcave part of the boundary of the target.  By the theorem of
Forstneri\v{c},
$f$ restricted to the inner ball
$r\B_n$ is a proper rational map to $R\B_N$ and gives the degree bound.
That $f$ is a proper map of $\B_n$ to $\B_N$ follows since we must take the
pseudoconvex part of the boundary to the pseudoconvex part.
\end{proof}

Via a theorem of Cima--Suffridge~\cite{cima-1990-boundary},
a rational proper
map of balls has no poles on the spheres, so we get a sphere map.
Conversely, a rational map taking a sphere to a sphere is a proper ball map 
via maximum principle.  In \cite{helal-2025-proper}, it was proved that a
rational sphere map also takes the outside of the sphere to the outside of
the sphere.  Putting all these together one gets the following equivalence.

\begin{corollary}
Suppose $n \geq 2$.
An 
$f \colon \A_{n,r} \to \A_{N,R}$ is a proper holomorphic map
if and only if $f$ is the restriction of a rational map
$f \colon \C^n \dashrightarrow \C^N$ with no poles on
the spheres $rS^{2n-1}$ nor $S^{2n-1}$ such that
$f(rS^{2n-1}) \subset RS^{2N-1}$ and
$f(S^{2n-1}) \subset S^{2N-1}$.
\end{corollary}

If $n=N=1$ the rationality still holds (but not the bound for the degree) as
we just showed above.  However, if $N > n=1$, then nonrational maps exist.

\begin{example}
Next, let us show that a nonrational proper map exists from $\A_{1,r}$ to
$\A_{2,R}$ for some $R$.  First, we construct a nonrational proper map from
the disc $\D$ to the 2-ball $\B_2$.
Such a map follows from the result of Dor~\cite{dor-1990-proper}, but it is useful
to consider an elementary construction in the one-dimensional case.
Take an arbitrary nonrational
function from the disc $f_1 \colon \D \to \D$, suppose that
it extends continuously to the closed disc $\overline{\D}$ but does not
extend holomorphically past any point on $\partial \D$ and suppose that
$\sup_{z\in\overline{\D}} \abs{f_1(z)} < 1$.  Find a real-valued
harmonic function $u$ on $\D$ continuous on $\overline{\D}$ so that
on $\partial \D$ we have
$u(z) = \frac{1}{2} \log \left(1-\abs{f_1(z)}^2\right)$.
Find its harmonic conjugate $v \colon \D \to \R$ and write
$f_2 = e^{u+iv}$.
Then the map $f = f_1 \oplus f_2$ is a proper map from $\D$
to $\B_2$ because $\abs{f_1(z)}^2+\abs{f_2(z)}^2 = \abs{f_1(z)}^2+e^{2u} \to 1$
as $\abs{z} \to 1$.
If $f_1$ is not rational, then
$f$ is not spherically equivalent to a rational map.

The map $f$ is at most points locally an embedding.  We can compose with
an automorphism on the source and target and assume that $f(0)=0$,
$f'(0)\not= 0$, and that for some neighborhood $U \subset \D$ of $0$
and a neighborhood $V \subset \B_2$ of $0$, $f(U) \subset V$ is a properly
embedded submanifold, $f|_U$ is one-to-one, and $f(\D \setminus U) \cap V =
\emptyset$.  Assume that $V$ is a ball.
Consider a closed ball $B$ centered at the origin that is small enough
so that
$\overline{B} \subset V$ and so that $K = f^{-1}(B)$ is connected.
The set $\D \setminus K$ is biholomorphic to some $\A_{1,r}$ for some $r$
via some $\varphi \colon \A_{1,r} \to \D \setminus K$.
Then $f \circ \varphi$ is a proper map of $\A_{1,r}$ to $\B_2 \setminus B =
\A_{2,R}$ (where $R$ is the radius of $B$).

The set $\D \setminus K$ has real-analytic boundary and so $\varphi$
extends past the boundary.
Therefore,
$f \circ \varphi$ cannot extend past the boundary at any point
on $\partial \D$
by the assumption on $f_1$ and so $f \circ \varphi$ is not rational.
\end{example}

Let us compare the effect of the automorphism groups
for proper holomorphic maps
$f \colon \A_{n,r} \to \A_{N,R}$ versus
$f \colon \B_n \to \B_N$.  The automorphism group of $\B_n$, $\Aut(\B_n)$,
is the set of linear fractional transformations of the form
\begin{equation} \label{eq:autBn}
U \phi_a(z) = 
U \frac{a-L_a z}{1-\langle z,a\rangle}
\qquad
L_a z = \left(1-\sqrt{1-\norm{a}^2}\right)
\frac{\langle z,a\rangle}{\norm{a}^2} a 
+
\sqrt{1-\norm{a}^2} z.
\end{equation}
For any affine subspace $L$ of dimension $m < n$ such that $L \cap \B_n \not=
\emptyset$, there exists a $\varphi \in \Aut(\B_n)$ so that $\varphi(L)$
is given by $z_{m+1} = \cdots = z_n = 0$.

The automorphism group of $\A_{n,r}$ is the unitary
group.  In particular, given an affine subspace $L$ of dimension $m < n$ that intersects both
the outer ball $\B_n$ and the inner ball $r\B_n$, elementary geometry says
that there is a $\varphi \in \Aut(\A_{n,r}) = U(n)$ so that $\varphi(L)$ is given by
$z_{m+1} = c$ and $z_{m+2} = \cdots = z_n = 0$ where $0 \leq c < r$.

The effect of this difference on classification up to equivalence with
respect to the embedding dimension $N_f$ is the following.
For a proper ball map $f \colon \B_n \to
\B_N$ where $N_f < N$, there exists an automorphism $\psi \in \Aut(\B_N)$
so that $\psi \circ f = F \oplus 0$ where $F \colon \B_n \to
\B_{N_f}$ is a proper map ($N_f = N_F$).  The corresponding result for annulus maps is
somewhat different, as an automorphism of the annulus is a unitary and
cannot take an affine space to a linear space; it can only ``rotate'' it.

\begin{proposition} \label[proposition]{prop:Nfsimplify}
Suppose $f \colon \A_{n,r} \to \A_{N,R}$, $n \geq 2$, is a proper map where $N_f < N$.
Then there exists a unitary $U \in U(N)$
and a unique $c \in [0,R)$ such that
\begin{equation}
U f = \sqrt{1-c^2}\, F \oplus c \oplus 0 ,
\end{equation}
where
$F \colon \A_{n,r} \to \A_{N_f,\sqrt{\frac{R^2-c^2}{1-c^2}}}$ is a proper map
and $N_f = N_F$.
\pagebreak[2]
\end{proposition}

\begin{proof}
The image $f(\A_{n,r})$ lies in an $N_f$-dimensional affine space $L$.
By \cref{theorem:annulimapsrational},
$f$ is rational and takes the two boundary spheres on the source to the two
boundary spheres on the target.
We compose $f$ with a unitary $U$
so that the space $L$ is given by $z_{N_f+1} = c$ for some $c \geq 0$
and $z_{N_f+2} = \cdots = z_{N} = 0$.  That is, $Uf = \widetilde{F} \oplus c
\oplus 0$ where $\widetilde{F} \colon \C^n \dashrightarrow \C^{N_f}$.
As $f$ takes the inner sphere to the inner sphere, we must have
that $L$ intersects the $R$-sphere and hence $c < R$.
The space $L$ intersected with the $R$-sphere is a sphere of radius
$\sqrt{R^2-c^2}$ and $L$ intersected with the $1$-sphere is a sphere 
of radius $\sqrt{1-c^2}$.
Hence $\widetilde{F}$ takes the $1$-sphere to the $\sqrt{1-c^2}$-sphere
and the $r$-sphere to the $\sqrt{R^2-c^2}$-sphere.  Dividing $\widetilde{F}$ by
$\sqrt{1-c^2}$, we obtain the map $F$ that takes the $1$-sphere to the $1$-sphere
and the $r$-sphere to the $\sqrt{\frac{R^2-c^2}{1-c^2}}$-sphere.
\end{proof}

In the proposition, the $0$ in the direct sum has $N-N_f-1$ components (so it
can be left out if $N=N_f+1$).

To classify annulus maps the following division lemma is very useful.
If $f = \frac{p}{q}$ is a rational proper annulus map, then it takes
the unit sphere to the unit sphere and
the $r$-sphere to the $R$-sphere.
Namely,
$\norm{p(z)}^2-\abs{q(z)}^2$ is divisible by $\norm{z}^2-1$ and
$\norm{p(z)}^2-R^2\abs{q(z)}^2$ is divisible by $\norm{z}^2-r^2$.
The following lemma encapsulates both of these conditions.
The lemma is a special case of Lemma 3.11 from \cite{helal-2025-proper}.

\begin{lemma}
\label[lemma]{lemma:quotientlemma}
Suppose $f \colon \A_{n,r} \to \A_{N,R}$ is a proper rational map of degree
$d$, $n \geq 2$, and write $f = \frac{p}{q}$ in lowest terms.  Then
$q$ is of degree at most $d-1$ and
\begin{equation}
\norm{p(z)}^2 = \Bigl(1 + b \bigl(\norm{z}^2-1\bigr)\Bigr) \abs{q(z)}^2 
+ Q(z,\bar{z}) \bigl(\norm{z}^2-1\bigr) \bigl(\norm{z}^2-r^2\bigr)
\end{equation}
for a real polynomial $Q(z, \bar{z})$ and $b = \frac{1-R^2}{1-r^2}$.
The $Q$ is of degree at most $d-2$ in $z$ and $d-2$ in $\bar{z}$.
\end{lemma}

Let us use the lemma to classify maps of degree 1.  We consider
the equidimensional case.
In the proof we will use the notion of a rank of a
real polynomial $\rho(z,\bar{z})$ in $\C^n$.  By this rank we mean
the rank of its matrix of coefficients, or alternatively the smallest number
$k$ so that
\begin{equation}
\rho(z,\bar{z}) = \sum_{j=1}^k \epsilon_j \abs{P_j(z)}^2 ,
\end{equation}
where $\epsilon_j = \pm 1$ and $P_j$ are holomorphic polynomials.

\begin{proposition} \label[proposition]{prop:deg1unitary}
If $f \colon \A_{n,r} \to \A_{n,R}$ is a proper rational map of degree 1
and $n \geq 2$,
then $r=R$ and $f$ is a unitary map.
\end{proposition}

\begin{proof}
We apply \cref{lemma:quotientlemma} to find $b$ and $Q$.  As $f$ is of
degree 1, then $Q\equiv 0$ and $q$ is a constant, meaning we can just take
$q=1$.
The rank of $\norm{f(z)}^2 = \norm{p(z)}^2$ must be the target dimension $n$, but
the rank of $\norm{p(z)}^2 = (1-b) + b \norm{z}^2$
is $n+1$ unless $b=1$.  So $b=1$ and $\norm{f(z)}^2 = \norm{z}^2$, which
implies that $f(z) = Uz$ for a unitary $U$.
\end{proof}

%\begin{proof}
%We have that $f$ is an LFT that must take the unit sphere to the unit
%sphere, hence it is well-known that it must have the form
%\begin{equation} \label{eq:autBn}
%U \phi_a(z) = 
%U \frac{a-L_a z}{1-\langle z,a\rangle}
%\qquad
%L_a z = \left(1-\sqrt{1-\norm{a}^2}\right)
%\frac{\langle z,a\rangle}{\norm{a}^2} a 
%+
%\sqrt{1-\norm{a}^2} z,
%\end{equation}
%where $U$ is a unitary and $a \in \B_n$.
%We need to show that $a=0$.
%By precomposing with a unitary, we can
%assume that $a = (a_1,0,\ldots,0)$ and
%\begin{equation}
%\phi_a(z) = 
%\left(
%\frac{a_1-z_1}{1-z_1\bar{a}_1}
%,
%\frac{-\sqrt{1-\abs{a_1}^2} \, z_2}{1-z_1\bar{a}_1},
%\ldots
%,
%\frac{-\sqrt{1-\abs{a_1}^2} \, z_n}{1-z_1\bar{a}_1}
%\right) .
%\end{equation}
%We note that the complex line $z_2=\cdots=z_n=0$ gets taken by $\phi_a$
%to this same complex line, and on this line $\phi_a$ is simply
%the corresponding LFT in one dimension $\frac{a_1-z_1}{1-z_1\bar{a}_1}$.
%
%It remains to show that if $z_1 \mapsto \frac{a_1-z_1}{1-z_1\bar{a}_1}$
%takes a circle of radius $r < 1$ centered at zero
%to a circle of radius $R < 1$ centered at zero.
%In other words, we must have that for all $\theta$,
%$\abs{a_1-r e^{i\theta}}^2 = R^2 \abs{1-re^{i\theta}\bar{a}_1}^2$.
%By expanding and equating coefficients one obtains that $a_1=0$
%and $r=R$.
%In other words, $a=0$ and $f$ is a unitary.
%\end{proof}

As a quick corollary, we can remove the hypothesis of rationality and degree.
By Hartogs phenomenon and the maximum principle, we find that
a proper holomorphic map $f \colon \A_{n,r} \to \A_{n,R}$ extends to a
proper holomorphic map of $\B_n$ to itself.
That this map is a rational map of degree 1 then follows by
Alexander's theorem.

\begin{corollary} \label[corollary]{corollary:equidimdeg1}
If $f \colon \A_{n,r} \to \A_{n,R}$ is a proper holomorphic map
and $n \geq 2$,
then $r=R$ and $f$ is a unitary map.
\end{corollary}

When $N > n$, the result is a little more complicated but follows in the
same way.

\begin{proposition} \label[proposition]{prop:linearembeddings}
If $f \colon \A_{n,r} \to \A_{N,R}$ is a proper rational map of degree 1 and
$N > n \geq 2$,
then $R \geq r$ and $f$ is unitarily equivalent to the affine embedding
\begin{equation}
z \mapsto \sqrt{\frac{1-R^2}{1-r^2}} \, z \oplus
\sqrt{\frac{R^2-r^2}{1-r^2}} \oplus 0 .
\end{equation}
\end{proposition}

\begin{proof}
Again apply \cref{lemma:quotientlemma} to find $q=1$ and
$\norm{f(z)}^2 = \norm{p(z)}^2 = (1-b) + b \norm{z}^2$.
For this to be a squared norm we must have $b \leq 1$, which implies $R \geq r$.
The form then follows as above.
\end{proof}

%\begin{proof}
%Write $f = \frac{p}{q}$.  As $f$ takes the unit sphere to the unit sphere
%we have
%\begin{equation}
%\norm{p(z)}^2-\abs{q(z)}^2 = \lambda(\norm{z}^2-1)
%\end{equation}
%where $\lambda$ has to be a constant as $p$ and $q$ are of degree 1.
%Therefore, the rank of $\norm{p(z)}^2$
%has to be $n$ as that is the number of positive eigenvalues of the
%matrix of coefficients $\lambda(\norm{z}^2-1)$.  In other words,
%$N_f = n$ and we can apply \cref{prop:Nfsimplify}.  The $F$ in the
%proposition is unitarily equivalent to the identity by
%\cref{prop:deg1unitary}.
%The values of the factors follow by elementary computation.
%\end{proof}

% ==============================================================

\section{General hyperplane rank and its connection to the \texorpdfstring{$X$}{X}-variety}
\label[section]{sec:kf}

Before proving some basic results about $k_f$ that we will require, let
us discuss some of the history of this invariant.
When Dor \cite{dor-1990-proper} constructed a nonrational map of balls,
to show nonrationality, he showed that (in our language) $k_f = N$,
as any rational ball map has $k_f \leq N-1$.
When $k_f = N-1$, we obtain a map of Grassmannians that takes an open
dense set of affine hyperplanes in $\C^n$ to affine hyperplanes in $\C^N$.
This technique was used by Faran to classify
proper maps from $\B_2$ to $\B_3$.
The number $k_f$ is also related to the $X$-variety that Forstneri\v{c} used
to show the rationality result.  Let us quickly define the $X$-variety
and note the connection for rational maps.
Let $f \colon \B_n \to \B_N$ be a rational proper map.
By complexifying the relation that $\norm{f(z)}^2=1$ whenever
$\norm{z}^2=1$,
we find $\langle f(z) , f(w) \rangle = 1$
whenever $\langle z , w \rangle =1$.  The $X$-variety is defined as
\begin{equation}
X_f =
\{ (z,\zeta) \in \C^n \times \C^N
: \langle \zeta , f(w) \rangle = 1 \text{ for all $w$ such that }
\langle z , w \rangle = 1
\} .
\end{equation}
Note that the equation for (almost) every affine hyperplane in $\C^n$
can be written as $\langle z , w \rangle = 1$ for a fixed $w$.
Moreover, $\langle \zeta, f(w) \rangle = 1$ gives affine hyperplanes on
the target.
The \emph{fiber} of the $X$-variety is the subset for a fixed $z$.
Such a fiber is an affine subspace of some dimension.
It is not difficult to show that if the generic fiber
of the $X$-variety has dimension $\ell$, then $k_f = N - \ell - 1$.
The $X$-variety encodes everything about the map $f$.  For example, if $\ell=0$,
then we find that the $X$-variety is (except some possible
exceptional fibers) the graph of the map $f$.
It also encodes the multivariable analogue of reflection across the sphere.
The $X$-variety was defined by Forstneri\v{c} for the rationality result,
and it has been extensively studied by D'Angelo,
see~\cites{dangelo-2003-Xvariety,dangelo-2007-asian}.
In fact, due to the connection of $k_f$ and the dimension of the fiber,
some of the basic results about the $k_f$ integer follow directly
from the work of D'Angelo, although we present direct proofs here.
In particular D'Angelo developed a technique of quickly computing the
fibers of the $X$-variety by associating to the map a certain matrix $C(w)$
whose nullspace gives the fiber of the $X$-variety at each point.
Bounding the general hyperplane rank using Green's hyperplane restriction
theorem from commutative algebra was also used in a paper by Grundmeier, the
second author, and Vivas \cite{grundmeier-2014-rank} to prove a rigidity
theorem for CR maps between hyperquadrics.

Let us then derive a few basic results about the $k_f$ integer.
First we note that for any rational map, it is enough to check an
open dense subset of hyperplanes, that is, the general hyperplane.  A similar
but more complicated statement
holds for holomorphic maps defined in some domain, but
we have no need for the complication.
The space of affine hyperplanes in $\C^n$
has the topology of the
Grassmannian (minus the hyperplane at infinity).

\begin{proposition}
\label[proposition]{prop:rank-dense-set}
If $f \colon \C^n \dashrightarrow \C^N$ is a rational map,
then
$\dim \aff\bigl(f(H)\bigr) = k_f$ for an open dense
set of affine hyperplanes $H$.
\end{proposition}

The proof is to simply note that the affine dimension
of $f(H)$ must be lower
semicontinuous on the space of hyperplanes $H$,
since the set of hyperplanes where the affine dimension
is at most some dimension is defined by polynomial
equations.

Next,
we recall the proof of
the classic result that $k_f \leq N-1$ for rational maps of
spheres, as the technique will be useful later.

\begin{proposition} \label[proposition]{prop:kfNm1}
Suppose $f \colon \C^n \dashrightarrow \C^N$ is a rational map so that
$f(S^{2n-1}) \subset S^{2N-1}$.  Then $k_f \leq N-1$.
\end{proposition}

\begin{proof}
Let $\langle \cdot,\cdot \rangle$ denote the standard sesquilinear
inner product so that the sphere $S^{2n-1}$ is defined by
$\norm{z}^2 = \langle z,z \rangle = 1$.
For any fixed $w \in \C^n$,
define the Segre variety
\begin{equation}
\Sigma_w = \{ z \in \C^n : \langle z,w \rangle = 1 \} .
\end{equation}
Similarly if $\omega \in \C^N$, we have the
Segre variety of the target sphere $\Sigma_\omega \subset \C^N$.
Note that the Segre variety for a sphere is an affine
hyperplane.
If $f=\frac{p}{q}$, then the fact that $f$ takes the sphere to
the sphere means that
$\norm{p(z)}^2 - \abs{q(z)}^2 = g(z,\bar{z}) \left( \norm{z}^2-1 \right)$,
or
$\norm{f(z)}^2 - 1 = \frac{g(z,\bar{z})}{\abs{q(z)}^2} \left( \norm{z}^2-1 \right)$,
where $g$ is a polynomial.
By polarizing this identity,
we find the equality
\begin{equation}
\inner{f(z),f(w)} - 1 = \frac{g(z,\bar{w})}{q(z)\overline{q(w)}} \left( 
\inner{z,w}-1
\right) .
\end{equation}
Writing $P = q^{-1}(0)$ for the pole set of $f$,
we therefore have that 
$f(\Sigma_w \setminus P) \subset \Sigma_{f(w)}$
for all $w$ where $f(w)$ is defined.
Any affine hyperplane that does not go through the origin
can be written as
$\{ z : \langle z,w \rangle=1 \}$ for some $w \in \C^n$.
Therefore, there is an open set of affine hyperplanes
which $f$ takes into hyperplanes.  Therefore $k_f \leq N-1$.
\end{proof}

Both the general hyperplane rank and the embedding dimension are related to the 
% span of components of the map together with the constant,
affine span
$\mathcal{L}_f = \Span_{\C} \set{1, f_1, \dots, f_N}$.
The following proposition restates the definitions in terms of $\mathcal{L}_f$.
\begin{proposition} \label[proposition]{prop:span}
    Let $f \colon U \subset \C^n \to \C^N$ be any map.
    Then
    \begin{enumerate}
        \item $N_f = \dim \mathcal{L}_f - 1$,
        \item $k_f = \max
        \set{N_{f|_{H \cap U}} \suchthat H \subset \C^n 
        \text{ is an affine hyperplane and } 
        H \cap U \not= \emptyset}$,
        \item $k_f = \max
        \set{\dim \mathcal{L}_{f|_{H \cap U}} - 1 \suchthat H \subset \C^n 
        \text{ is an affine hyperplane and } 
        H \cap U \not= \emptyset}$.
    \end{enumerate}
\end{proposition}%

As linear fractional transformations preserve affine subspaces,
the general hyperplane rank $k_f$ is also invariant under conjugation by LFTs.

\begin{proposition}
\label[proposition]{prop:spherical-invariant}
For a map $f \colon U \subset \C^n \to \C^N$,
$\phi$ an LFT on $\C^n$,
and $\psi$ an LFT on $\C^N$, we have
$k_f = k_{\psi \circ f \circ \phi}$.
In other words, $k_f$ is both a spherical and a unitary invariant.
\end{proposition}

As part of 
\cref{thm:kfN1impliesHd}, we claim that $k_{H_d} = N-1$.  Let us prove that
now.

\begin{proposition}
\label[proposition]{prop:kfHd}
For the homogeneous map $H_d$, we have
$k_{H_d} = N-1$ where $N = \binom{n+d-1}{d}$ is the target dimension of
$H_d$.
\end{proposition}

\begin{proof}
The map $H_d$ has as its components every degree $d$ monomial in $n$
variables, $H_d = (z_1^d, \sqrt{d} \, z_1^{d-1}z_2, \ldots)$.
If we consider what $H_d$ does to the hyperplane $\{ z_1 = c \}$
for nonzero constants $c$, we find that except for the first component
which becomes the constant $c^d$, all the other components are affine
independent.  In other words, the affine span of $H_d(\{z_1=c\})$ is of
dimension $N-1$.
\end{proof}

Next we consider the $k_f$ integer for juxtapositions.
We can think of the set of rational ball (or annulus) maps
as a convex set; the set of the squared norms
$\norm{f(z)}^2$
where $f$ is a map of balls (or annuli) is in fact convex,
and the juxtaposition of $f$ and $g$ has the squared norm
$(1-t)\norm{f(z)}^2+ t\norm{g(z)}^2$
that is, it is a convex combination,
see for example~\cite{dangelo-2009-complexity}.
The next proposition therefore
shows that for rational ball (or annulus) maps,
$k_f = N_f-1$ happens only on the non-extremal points
of this convex set.

Define 
$\rho_f(z, \bar{z}) = \norm{f(z)}^2 - 1 = \inner{f(z), f(z)}-1$,
which polarizes to
$\rho_f(z, \bar{w}) = \inner{f(z), f(w)}-1$.

\begin{proposition}
\label[proposition]{proposition:juxtaposekf}
Suppose that $F_t$ is a juxtaposition of two rational proper maps
of balls (or annuli) $F_0$ and $F_1$,
that is, $F_t = \sqrt{1-t}\, F_0 \oplus \sqrt{t}\, F_1$
for some $t \in (0,1)$.
%Exactly one of the following holds.
\begin{enumerate}
    \item If $\rho_{F_0} = \norm{F_0(z)}^2 - 1$ and $\rho_{F_1} = \norm{F_1(z)}^2 - 1$ 
    are linearly independent, then
    $k_{F_t} \leq N_{F_t}-2$.

    \item If $\rho_{F_0}$ and $\rho_{F_1}$ are linearly dependent, then
    $F_0$, $F_1$ and $F_t$ all have the same invariants
    $k_{F_0} = k_{F_1} = k_{F_t}$ and $N_{F_0} = N_{F_1} = N_{F_t}$.
\end{enumerate}
\end{proposition}

In the first item, the condition on the norms says that $F_0$ and $F_1$ do not differ by a
    juxtaposition with a constant map followed by a target unitary.
    It will also follow from the proof that
    if the map $F_j$ has exactly $N_{F_j}$ components for both $j$, that is,
    both target dimensions are embedding dimensions, then
    the condition is equivalent to $\norm{F_0(z)}^2 \not\equiv \norm{F_1(z)}^2$.

\begin{proof}
 Let us prove the first item. 
    By \cref{prop:rank-dense-set}, $k_{F_t}$ is achieved by generic hyperplanes.
    For generic $w$, the Segre variety $\Sigma_w$ is a generic hyperplane.
    We claim that $\dim \aff F_t(\Sigma_w) \leq N_{F_t} - 2$,
    so that $k_{F_t} \leq N_{F_t} - 2$.

    Since $F_t$ is a juxtaposition of $F_0$ and $F_1$,
    we have 
    \begin{equation}
        \mathcal{L}_{F_t} 
        = \mathcal{L}_{F_0} + \mathcal{L}_{F_1}
        = \Span_\C \set{1, F_{0, 1}, \dots, F_{0, N_0}, F_{1, 1}, \dots, F_{1, N_1}},
    \end{equation}
    where $F_j$ maps into $\C^{N_j}$ for both $j$.
    Since for a fixed $w$,
    \begin{equation}
        \rho_{F_j}(\cdot, \bar{w}) 
        = -1 + \inner{F_j(\cdot), F_j(w)} 
        = -1 + \sum_{k=1}^{N_j} \overline{F_{j, k}(w)} F_{j, k},
    \end{equation}
    both $\rho_{F_j}(\cdot, \bar{w})$ belong to $\mathcal{L}_{F_t}$.

    We will now show that 
    $\rho_{F_0}(\cdot, \bar{w})$ and $\rho_{F_1}(\cdot, \bar{w})$ 
    are linearly independent for generic $w$.
    Let $N=\dim \mathcal{L}_{F_t}$ and
    choose a basis $\set{v_1, \dots, v_N}$ for $\mathcal{L}_{F_t}$ and write
    \begin{equation}
        \rho_{F_j}(\cdot, \bar{w}) = \sum_{k=1}^{N} r_{jk}(\bar{w}) v_k
    \end{equation}
    for both $j$.
    The coefficient functions $r_{jk}$ are rational and antiholomorphic in $w$.
    Therefore, $\rho_{F_0}(\cdot, \bar{w})$ and $\rho_{F_1}(\cdot, \bar{w})$ 
    are linearly independent if and only if at least one of the $2 \times 2$ minors
    \begin{equation} \label{eq:minors}
        r_{0k}(\bar{w}) r_{1\ell}(\bar{w}) - r_{0\ell}(\bar{w}) r_{1k}(\bar{w})
    \end{equation}
    is nonzero.
    Suppose for contradiction that all these minors vanish identically.
    Then $\rho_{F_0}(\cdot, \bar{w})$ and $\rho_{F_1}(\cdot, \bar{w})$ 
    are linearly dependent for every $w$ where they are defined.
    Evaluating $F_1$ at any point $p$ on the source ball,
    $\rho_{F_1}(p, \bar{p}) = \norm{F_1(p)}^2 - 1 < 0$,
    so 
    $\rho_{F_1}(z, \bar{w}) \neq 0$
    for all $z, w \in V$ near a small enough neighborhood $V$ of $p$.
    The linear dependence gives us
    \begin{equation}
        \rho_{F_0}(z, \bar{w}) = c(\bar{w}) \rho_{F_1}(z, \bar{w})
    \end{equation}
    for some nonzero function $c$ for every $w \in V$ for every $z$.
    Using the hermitian symmetry and interchanging $z$ and $w$, we get
    that for $z$ and $w$ in $V$, we have
    $\rho_{F_0}(z, \bar{w}) = \overline{c(\bar{z})} \rho_{F_1}(z, \bar{w})$,    
    so that 
    \begin{equation}
        \bigl(c(\bar{w}) - \overline{c(\bar{z})}\bigr) \rho_{F_1}(z, \bar{w}) = 0,
    \end{equation}
    that is, $c(\bar{w}) - \overline{c(\bar{z})} = 0$ on $V$.
    Therefore, $c$ must be a nonzero constant and real on $V$, and so
    \begin{equation} \label{eq:rhoF0rhoF1dep}
        \rho_{F_0}(z, \bar{w}) = c \rho_{F_1}(z, \bar{w})
    \end{equation}
    for every $w$ in the nonempty open set $V$.
    The identity theorem on the rational functions $\rho_{F_j}(z, \bar{w})$
    gives us that \cref{eq:rhoF0rhoF1dep} holds for all $z$ and $w$
    outside the poles of $F_t$.
    The diagonal $z = w$ gives us 
    $\rho_{F_0}(z, \bar{z}) \equiv c \rho_{F_1}(z, \bar{z})$,
    contradicting that 
    $\rho_{F_0}(z, \bar{z})$ and $\rho_{F_1}(z, \bar{z})$ are linearly independent.
    Therefore, at least one of all the minors \cref{eq:minors} does not vanish identically.
    So for generic $w$,
    that minor is nonzero and hence
    $\rho_{F_0}(\cdot, \bar{w})$ and $\rho_{F_1}(\cdot, \bar{w})$ 
    are linearly independent.

    As on the Segre variety $\Sigma_w$, 
    $\rho_{F_j}(\cdot, \bar{w}) = 0$ for both $j$,
    let us consider the onto restriction map
    \begin{equation}
        R_w \colon \mathcal{L}_{F_t} \to \mathcal{L}_{F_t}|_{\Sigma_w}.
    \end{equation}
    Then $\rho_{F_0}(\cdot, \bar{w})$ and $\rho_{F_1}(\cdot, \bar{w})$
    belong to $\ker R_w$ and are linearly independent, so that
    \begin{equation}
        \dim \ker R_w \geq 2,
    \end{equation}
    and by the rank-nullity theorem and \cref{prop:rank-dense-set,prop:span},
    \begin{equation}
    \begin{split}
        k_{F_t}
        &= \dim \aff F_t(\Sigma_w) \\
        &= \dim \mathcal{L}_{F_t(\Sigma_w)} - 1 \\
        &= \dim \mathcal{L}_{F_t} - \dim \ker R_w - 1 \\
        &= N_{F_t} - \dim \ker R_w \\
        &\leq N_{F_t} - 2.
    \end{split}
    \end{equation}

Let us now prove the second item.
    We have $\rho_{F_0}(z, \bar{z}) = c \rho_{F_1}(z, \bar{z})$ 
    for some nonzero constant $c$.
    Evaluating $F_j$ at a point $p$ on the source ball,
    $\norm{F_j(p)}^2 < 1$ for both $j$.
    This means that both $\rho_{F_0}$ and $\rho_{F_1}$ are negative at $p$ and $c$ is a positive real constant.

    If $0 < c \leq 1$, then 
    \begin{equation}
        \norm{F_0}^2 = c \norm{F_1}^2 + (1-c) = \norm{\sqrt{c} F_1 \oplus \sqrt{1-c}}^2,
    \end{equation}
    so that $F_0$ equals $\sqrt{c} F_1 \oplus \sqrt{1-c}$ up to a target unitary.
    Therefore,
    \begin{equation}
        \mathcal{L}_{F_0} = \Span_{\C} \set{1, F_{0,1}, \dots, F_{0,N}} = \mathcal{L}_{F_1}.
    \end{equation}
    If on the other hand $c > 1$, then we write $\rho_{F_1} = \frac{1}{c} \rho_{F_0}$
    with $0 < \frac{1}{c} < 1$, and so by the previous case, 
    $\mathcal{L}_{F_1} = \mathcal{L}_{F_0}$.

    Since $F_t$ is a juxtaposition of $F_0$ and $F_1$,
    we have 
    $\mathcal{L}_{F_t} 
    = \mathcal{L}_{F_0} + \mathcal{L}_{F_1} 
    = \mathcal{L}_{F_0} 
    = \mathcal{L}_{F_1}$.
    The conclusion follows from \cref{prop:span}.
    Note that if the map $F_j$ has exactly $N_{F_j}$ components for both $j$,
    then $c$ is forced to be $1$, that is, $\norm{F_0(z)}^2 = \norm{F_1(z)}^2$.
\end{proof}

\section{First gap} \label[section]{sec:gap}

Before we prove \cref{thm:gapthm1},
we need a result on (hermitian) rank that uses similar 
statement and proof from \cite{dangelo-2012-pfisters}*{Theorem~1.1}.
We only need the quadratic version of this result, but we state the more
general result.
Recall that the rank of
a real polynomial $r(z,\bar{z})$ is
the rank of its matrix of coefficients, or alternatively the smallest $k$
such that
$r(z,\bar{z}) = \sum_{j=1}^k \epsilon_j \abs{P_j(z)}^2$.

% --------------------------------------------------------------
\begin{theorem}
\label[theorem]{thm:hermitian}
Let $n \geq 1$,
$s_1, s_2$ be nonzero real numbers of the same sign,
and $Q(z, \bar{z})$ be a nonzero real polynomial on $\C^n$.
Then
\begin{equation}
\rank \bigl( Q(z,\bar{z}) (\norm{z}^2 + s_1) (\norm{z}^2 + s_2) \bigr)
\geq \rank \bigl( (\norm{z}^2 + s_1) (\norm{z}^2 + s_2) \bigr)
= \binom{n+2}{2}.
\end{equation}
\end{theorem}

The proof follows by roughly the same argument as in 
\cite{dangelo-2012-pfisters} which proved the special case
when $s_1=s_2=1$, but with an 
arbitrary number of factors.  What must be newly proved here
is a ``diagonal version'' of the theorem.
We only have this proof for two factors, which is sufficient for our purposes.
We conjecture that the result holds for any number of factors.
Note that the result is not true as stated if the numbers $s_j$ have opposite
signs.  For example,
$(\abs{z}^2+1)(\abs{z}^2-1) = \abs{z}^4-1$
is of rank 2 and not $\binom{n+2}{2} =\binom{3}{2} = 3$.

\begin{proof}
The main new claim that must be proved is the ``diagonal version'' of
this theorem.  That is, for real polynomials of
real variables $x_1,\ldots,x_n$, we have
\begin{multline}
    \# \bigl(
      Q(x) (x_1+\cdots+x_n + s_1) (x_1+\cdots+x_n + s_2) 
    \bigr)
\\
    \geq
    \# \bigl(
      (x_1+\cdots+x_n + s_1) (x_1+\cdots+x_n + s_2)
    \bigr)
    = \binom{n+2}{2},
\end{multline}
where $\#$ means the number of nonzero coefficients in the given
polynomials when they are expanded.
We note that if we prove the claim for positive $s_1,s_2$,
then it follows for negative $s_1,s_2$ or vice-versa by replacing
$x$ by $-x$, factoring out a $(-1)^2$ and applying the claim
for positive $s_1,s_2$.

First suppose that $n=1$.  For this case, assume $s_1$ and $s_2$
are both negative.
The polynomial
\begin{equation}
R(x) = Q(x) (x + s_1) (x + s_2) 
\end{equation}
has at least $2$ positive real roots.  By Descartes' rule of signs,
there must be at least $2$ sign changes among the coefficients of
the polynomial and therefore there must be at least $3$ nonzero
coefficients, which is precisely the claimed result for $n=1$.

Next, let us suppose $n=2$.
Again let $R(x) = Q(x) (x_1 + x_2 + s_1) (x_1 + x_2 + s_2)$.
Write $R(x) = x_1^j x_2^k \widetilde{R}(x)$ (and similarly
$Q(x) = x_1^j x_2^k \widetilde{Q}(x)$)
where $\widetilde{R}(x_1,0)$ and $\widetilde{R}(0,x_2)$ are both
not identically zero.
Applying
the one variable version, we find that both
$\widetilde{R}(x_1,0)$ and $\widetilde{R}(0,x_2)$
have at least $3$ monomials (which are pure powers of $x_1$ or $x_2$
respectively).  One
of these terms can be the constant, and therefore, we
have $5$ terms in $\widetilde{R}(x)$ which are all pure powers.
Next we will show that
there is a term in $\widetilde{R}(x)$ that involves both $x_1$ and $x_2$.
Note that the vector field
$L =\frac{\partial}{\partial x_1}-\frac{\partial}{\partial x_2}$
is tangent to each $(x_1+x_2+s_j)$.  Therefore, if we apply $L$ $\nu$ number
of times, we get
\begin{equation}
L^\nu \widetilde{R}(x) = 
\bigl(L^\nu \widetilde{Q}(x)\bigr)
(x_1+x_2 + s_1) (x_1+x_2 + s_2) .
\end{equation}
As $\widetilde{Q}$ is a polynomial, $L^\nu \widetilde{Q}$ must be
identically zero for some $\nu$, so take $\nu$ such that
$L^{\nu-1} \widetilde{Q}$
is not zero but
$L^{\nu} \widetilde{Q} \equiv 0$.
Then
$L^{\nu-1}\widetilde{Q}$ is a solution to the transport equation $Lu = 0$,
and so it is a function of $x_1+x_2$, that is,
$L^{\nu-1}\widetilde{Q}(x) = S(x_1+x_2)$ and by the choice of $\nu$,
$S$ is not identically zero.
But then
\begin{equation}
L^{\nu-1}\widetilde{R}(x) = 
S(x_1+x_2)
(x_1+x_2 + s_1) (x_1+x_2 + s_2)
\end{equation}
is a polynomial in $(x_1+x_2)$ of degree at least 2 and hence
it must have a mixed monomial.
Since $L^{\nu-1}\widetilde{R}(x)$ has a mixed monomial, then
so does $\widetilde{R}(x)$ as $L$ cannot manufacture mixed monomials out
of pure ones.
Therefore, $R(x)$ has at least $6 = \binom{2+2}{2}$ terms.
In particular, we have the following terms:

\begin{statement} \label[statement]{statement:distributionofterms}
In $R(x_1,x_2) = x_1^j x_2^k \widetilde{R}(x_1,x_2)$,
there are at least 3 terms with $x_1^jx_2^m$ for
$m \geq k$.  There are at least 3 terms with
$x_1^mx_2^k$ for $m \geq j$.  There is at least one term
with $x_1^\ell x_2^m$ for $\ell > j$ and $m > k$.
There are no terms with 
$x_1^\ell x_2^m$ for $\ell < j$ or $m < k$.
\end{statement}

We now attack the general problem for $n > 2$.
We proceed by induction on the dimension.
Suppose the claim is true for dimension $n-1$.
Let
$R(x) = Q(x) (x_1+\cdots+x_n+s_1) (x_1+\cdots+x_n+s_2)$ as before.
Without loss of generality suppose that $R$ (and hence also $Q$) is not
divisible by $x_1$.  Then via the induction hypothesis,
\begin{equation}
R(0,x_2,\ldots,x_{n}) = 
Q(0,x_2,\ldots,x_{n}) 
(x_2+\cdots+x_{n}+s_1) (x_2+\cdots+x_{n}+s_2)
\end{equation}
must have at least $\binom{n+1}{2}$ terms, and hence $R(x)$
has at least
$\binom{n+1}{2}$ terms that do not depend on $x_1$.  We need to show
that there must be at least $n+1$ terms in $R(x)$ that do depend on $x_1$.

We will prove this statement by induction again.
In two variables, we have proved by
\cref{statement:distributionofterms}
that there are at least 3 terms depending on $x_1$.
Now suppose that
in the problem with $n-1$ variables,
$R$ always contains $n$ terms depending on $x_1$.
Take $R(x)$ as above in $n$ variables.
Without loss of generality we can assume that $R$ is also not divisible by $x_2$.
The function $R(x_1,0,x_3,\ldots,x_n)$ satisfies our induction hypothesis
and hence contains at least $n$ distinct terms involving $x_1$.
That is, $R(x)$ contains $n$ distinct terms involving $x_1$ but not
depending on $x_2$.  We simply need to now find a term that depends on both
$x_1$ and $x_2$.
Write $x=(x_1,x_2,x')$ and fix one of the standard monomial orders on the remaining
variables $x'$, such that multiplying a monomial in $x'$
by any one of the variables $x_3,\ldots,x_n$ moves the monomial further
in the ordering.
Using multinomial notation, write
\begin{equation}
R(x) = \sum_{\alpha} r_{\alpha}(x_1,x_2) (x')^\alpha 
\quad\text{and} \quad
Q(x) = \sum_{\alpha} q_{\alpha}(x_1,x_2) (x')^\alpha .
\end{equation}
Find the smallest monomial $(x')^\mu$ in the given monomial order
for which the $q_\alpha$ is not identically zero.
We then have that
\begin{equation}
r_{\mu}(x_1,x_2) = q_\mu(x_1,x_2) (x_1+x_2+s_1)(x_1+x_2+s_2) .
\end{equation}
All the monomials of $r_\mu(x_1,x_2) (x')^\mu$ appear in $R(x)$.
By \cref{statement:distributionofterms}, there must be at least
one term that depends on both $x_1$ and $x_2$.

Therefore there are at least $n+1$ terms depending on $x_1$.
As $R$ had at least $\binom{n+1}{2}$ terms not depending on $x_1$,
we find that $R$ has at least $\binom{n+1}{2}+n+1 = \binom{n+2}{2}$
terms and the claim is proved by induction.

Now that we have the claim, the theorem follows in the same way
as in the proof of the polynomial version of Theorem 2.1 in
\cite{dangelo-2012-pfisters}.  That is, there exists some ordering
of monomials so that if we write out the
matrix of coefficients of the expressions involved,
multiplication by $\abs{z_j}^2$ moves along diagonals.  The rank of the
matrix of coefficients is the rank of the function.  The rank of a matrix
can be estimated from below by considering the number of nonzero terms
in the superdiagonal that is
furthest up and to the right that has nonzero terms.  Up to multiplication
with a monomial, the terms on this diagonal are polynomials
in $\abs{z_j}^2$, and so by replacing $x_j = \abs{z_j}^2$ we obtain
the situation of the claim and the proof follows.  See
\cite{dangelo-2012-pfisters} for more details.
\end{proof}

We can now prove
\cref{thm:gapthm1},
that is,
if $f \colon \A_{n,r} \to \A_{N,R}$ ($n \geq 2$) is a proper holomorphic map
and $n < N < \binom{n+1}{2}$, then
$f$ is unitarily equivalent to an affine embedding
\eqref{eq:linearembed}
and $r \leq R < 1$.

\begin{proof}[Proof of \cref{thm:gapthm1}]
    By \cref{theorem:annulimapsrational}, 
    $f$ is a rational map that takes 
    $S^{2n-1}$ to $S^{2N-1}$ and $r S^{2n-1}$ to $R S^{2N-1}$.
    Write $f = \frac{p}{q}$ in lowest terms.
    By \cref{lemma:quotientlemma} we have
    \begin{equation}
        \norm{p(z)}^2
        = \Bigl(1 + b \bigl(\norm{z}^2-1\bigr)\Bigr) \abs{q(z)}^2 
        + Q(z,\bar{z}) \bigl(\norm{z}^2-1\bigr) \bigl(\norm{z}^2-r^2\bigr)
    \end{equation}
    for a real polynomial $Q(z, \bar{z})$ and $b = \frac{1-R^2}{1-r^2}$.

    If $Q \equiv 0$, then since $\frac{p}{q}$ is in lowest terms, we must have
    $q$ is constant, so we may assume $q \equiv 1$.
    Then 
    $\norm{p(z)}^2 = 1 + b \bigl(\norm{z}^2-1\bigr)$
    is of bidegree $(1, 1)$ (degree 1 in $z$ and degree 1 in $\bar{z}$).
    Therefore, $\deg f = 1$.
    \cref{prop:linearembeddings} implies that $f$ is an affine embedding
    and that $R \geq r$.

    Suppose $Q \not\equiv 0$, we will get a contradiction.
    By applying \cref{prop:Nfsimplify}, we reduce to the case when
    $N=N_f$.
    The dimension $N$ is then the number of components of $f$, which is
    equal to the number of components of $p$, which we can assume are linearly
    independent as $N_f = N$.  Therefore, $N$ is equal to the rank
    of $\norm{p(z)}^2$.
    \cref{thm:hermitian} gives us
    \begin{equation}
    \begin{split}
    N &= \rank \norm{p(z)}^2 \\
        &\geq
        -\rank \Bigl(1 + b \bigl(\norm{z}^2-1\bigr)\Bigr) \abs{q(z)}^2 
        + \rank Q(z,\bar{z}) \bigl(\norm{z}^2-1\bigr) \bigl(\norm{z}^2-r^2\bigr) \\
        &\geq
        -\rank \Bigl(1 + b \bigl(\norm{z}^2-1\bigr)\Bigr) 
        + \rank \bigl(\norm{z}^2-1\bigr) \bigl(\norm{z}^2-r^2\bigr) \\
        &= -\rank \bigl( b \norm{z}^2 + (1 - b) \bigr)
        + \rank \bigl( \norm{z}^4 - (1 + r^2) \norm{z}^2 + r^2 \bigr) \\
        &\geq - \binom{n}{1} - \binom{n-1}{0}
        + \binom{n+1}{2} + \binom{n}{1} + \binom{n-1}{0}
        = \binom{n+1}{2}.
    \end{split}
    \end{equation}
    In the inequalities we have used the facts that
    $\rank (A+B) \geq - (\rank A) + \rank B$ and that
    $\rank \bigl(\abs{q(z)}^2 A(z,\bar{z})\bigr) = \rank A(z,\bar{z})$, both of which 
    follow immediately from the definition of rank.
    By hypothesis, we find a contradiction as we assumed
    $N_f < \binom{n+1}{2}$.

    We therefore obtained $N_f = n$ and that the degree is $1$.
    The conclusion follows from \cref{prop:deg1unitary} or
    \cref{prop:linearembeddings} depending on whether $N > n$
    or $N = n$.
    The interval is sharp; when 
    $N = \binom{n+1}{2}$, besides the affine embeddings we
    can also have the map $H_2$ (depending on $R$),
    which is clearly not unitarily
    equivalent to an affine embedding.
\end{proof}

We should remark that
for specific $\A_{n,r}$ and $\A_{\binom{n+1}{2},R}$ we either get
the affine embeddings if $R \geq r$, or we get the homogeneous
map $H_2$ if $R=r^2 < r$.  For general $n$, it is not clear if other
maps exist when $N=\binom{n+1}{2}$, for perhaps different $R$.
For $n=2$, the homogeneous $H_2$ and the affine embeddings
are the only proper maps from 
$\A_{2,r}$ to $\A_{\binom{2+1}{2},R}=\A_{3,R}$.
See \cref{theorem:farananalogue}.

% ==============================================================

\section{Homogeneous maps via the general hyperplane rank} \label[section]{sec:homog}

To prove the theorem on classification of homogeneous maps,
we will prove a more general result about rational maps
taking two spheres to spheres.  To state the result, it may be easiest to work
in the projective space $\bP^n$.  The automorphism group $\Aut(\bP^n)$ of $\bP^n$
is the set of linear fractional transformations or LFTs.  Consider
some embedding $\C^n \subset \bP^n$ and then we can consider the sphere
$S^{2n-1}$ as a subset of $\bP^n$.  When we refer to affine coordinates,
we refer to this embedding and the coordinates $(z_1,\ldots,z_n)$.
We say $\sS$ is an \emph{LFT sphere}
if there exists an LFT $\varphi \in \Aut(\bP^n)$ so that
$\sS = \varphi(S^{2n-1})$.

Our first result is about rational maps taking two LFT spheres to
LFT spheres.  The definition of the general hyperplane rank $k_f$ for a rational map
$f \colon \bP^n \dashrightarrow \bP^N$ in the projective setting is exactly
analogous to the affine setting; it is the smallest dimension $k$ of
a $k$-plane in $\bP^N$ that contains $f(H)$ for a general hyperplane
$H \subset \bP^n$.  In the proof we will work in homogeneous coordinates on
the projective space, and
so a general hyperplane in 
$H \subset \bP^n$ is an $n$-dimensional linear subspace in $\C^{n+1}$.

We use $z = (z_0,\ldots,z_n)$ for the homogeneous coordinates 
on $\bP^n$.
To write a rational map $f \colon \bP^n \dashrightarrow \bP^N$ in
homogeneous coordinates we start with $f = \frac{(p_1,\ldots,p_N)}{q}$
and order it as $(q,p_1,\ldots,p_N)$ and homogenize, that is, we make the
denominator the $0$th component.

The sphere $S^{2n-1} \subset \C^n \subset \bP^n$
in our chosen affine embedding when put into homogeneous coordinates
is given by the equation
\begin{equation}
\inner{J_n z,z}  = 0 , \quad \text{where} \quad
J_n =
\operatorname{diag}(-1,1,\ldots,1) =
\begin{bmatrix}
-1      & 0      & \cdots & 0      \\
 0      & 1      & \cdots & 0      \\
 \vdots & \vdots & \ddots & \vdots \\
 0      & 0      & \cdots & 1      \\
\end{bmatrix} .
\end{equation}
Here $\inner{\cdot,\cdot}$ is the standard sesquilinear inner product on 
$\C^{n+1}$.
As LFTs are invertible matrices,
any other LFT sphere can then be written as
\begin{equation}
\inner{T^*J_nT z,z}  = 0 ,
\end{equation}
for an invertible $T$, that is, it is $\inner{A z,z}  = 0$ for some
hermitian $(n+1) \times (n+1)$ matrix $A$ with $n$ positive and $1$ negative
eigenvalues.  If $\varphi$ is the LFT denoted by the matrix $T$,
then $\inner{T^*J_nT z,z}  = 0$ represents the sphere
$\varphi^{-1}(S^{2n-1})$.
We can also do the same in $\C^{N+1}$ on the target, using instead the
matrix $J_N$.

\begin{theorem} \label[theorem]{thm:lft}
Suppose $f \colon \bP^n \dashrightarrow \bP^N$ is a nonconstant rational map
that takes two distinct LFT spheres to LFT spheres and $k_f = N - 1$.
Then $f \circ \alpha = \beta \circ f$ 
for some LFTs $\alpha \in \Aut(\bP^n)$ and $\beta \in \Aut(\bP^N)$, neither
of which is the identity.

More precisely, using the coordinates above, if one of the spheres on the 
source (resp.\ the target) is given by $J_n$ (resp.\ $J_N$) and the other
by a hermitian matrix $A$ (resp.\ $B$), then $\alpha$ is
represented by the matrix $J_n A$ and $\beta$ is represented by
the matrix $J_N B$.
\end{theorem}

The main idea behind the proof is that $k_f = N-1$ means that the map $f$
induces a map of the corresponding Grassmannians.  To figure out the map one
uses the Segre variety of the LFT spheres.  But now we have
two LFT spheres that go to LFT spheres and the second LFT sphere pair means that $f$
induces a different map of
Grassmannians.  Comparing the two maps gives us the corresponding $\alpha$
and $\beta$.

\begin{proof}
A nonconstant rational sphere map gives a proper map of the balls.
In \cite{helal-2025-proper}, we showed that such a map
not only takes the inside of the ball to the inside of the ball,
but takes the complement of the closed ball (except for the poles) to
the complement of the closed ball.  This implies that the two target
spheres in this theorem must also be distinct: If $S_1$ and $S_2$ are
the distinct source LFT spheres and $X_1$ and $X_2$ are the target
LFT spheres, then points on the complement of $S_1$ are taken to
the complement of $X_1$, so $f(S_2)$ contains points of
the complement of $X_1$, and hence $X_1 \neq X_2$.

As we said we will work in
homogeneous coordinates, that is, in $\C^{n+1}$
and $\C^{N+1}$.
We will write $z$ and $w$ for points in $\C^{n+1}$ and $Z$ and $W$ for points in $\C^{N+1}$.
Starting with a rational map $f$, we clear denominators and
homogenize to represent $f$ by a homogeneous polynomial map
$F  = (F_0,F_1,\ldots,F_N) \colon \C^{n+1} \to \C^{N+1}$ where $F_0$
is the denominator.
We assume that $f$ was written in lowest terms and hence $F_0,\ldots,F_N$ have no common factor.
After applying an LFT we can assume that one of the matrices on
the source corresponds to the standard sphere $S^{2n-1}$, that is, it is given by
$\inner{J_n z,z}  = 0$, and similarly for the target.
By \cref{prop:spherical-invariant},
this does not change the $k_f$ or the hypothesis.
For any hermitian matrix $M$ and a point $w$, we write the
Segre variety of the homogenized LFT sphere given by $\inner{Mz,z}=0$ as
\begin{equation}
\Sigma_w^{M} = \{ z : \inner{Mz,w} = 0 \} .
\end{equation}

Write $z = (z_0,z')$ where $z'$ refers to the affine coordinates.
Then the fact that $f = \frac{p}{q}$ takes the standard sphere to the standard sphere means
$\inner{p(z'),p(w')} - q(z')\overline{q(w')} = g(z',\bar{w}') \left( \inner{z',w'}-1 \right)$
for a polynomial $g$.  In homogeneous coordinates, this translates to
\begin{equation}
\inner{J_N F(z) , F(w)} =
G(z,\bar{w}) \inner{J_nz,w} .
\end{equation}
This equation implies that
\begin{equation}
F(\Sigma_w^{J_n}) \subset \Sigma_{F(w)}^{J_N} .
\end{equation}
We are assuming that
we have a homogenized LFT sphere in $\C^{n+1}$ represented by the
$(n+1) \times (n+1)$ hermitian matrix $A$
that $F$ takes to a homogenized LFT sphere in $\C^{N+1}$ represented
by the
$(N+1) \times (N+1)$ hermitian matrix $B$.  Therefore, we also have
\begin{equation}
F(\Sigma_w^{A}) \subset \Sigma_{F(w)}^{B} .
\end{equation}
Let 
$S = (AJ_n^{-1})^* = J_nA$.
As
\begin{equation}
\inner{Az,w} = \inner{AJ_n^{-1}J_n z,w} =
\inner{J_n z,(AJ_n^{-1})^*w} ,
\end{equation}
we find that $\Sigma_w^A=\Sigma_{Sw}^{J_n}$, and so
\begin{equation}
F(\Sigma_w^{A}) \subset \Sigma_{F(w)}^{B}
\quad \text{and} \quad
F(\Sigma_w^{A}) = F(\Sigma_{Sw}^{J_n}) \subset \Sigma_{F(Sw)}^{J_N} .
\end{equation}
The condition that $k_f = N-1$ means in homogeneous coordinates
that given a generic linear $n$-dimensional subspace $V \subset\C^{n+1}$,
the linear span of $F(V)$ in $\C^{N+1}$ is $N$-dimensional.
Let $\mathcal{G}$ be the set of such $V$, and put
\begin{equation}
    \mathcal{E} = 
    \set{w \in \C^{n+1} \suchthat F(w) \neq 0, F(Sw) \neq 0, \Sigma_w^A \in \mathcal{G}}.
\end{equation}
The set $\mathcal{E}$ is open and dense as $F$ is not identically zero,
$S$ is invertible,
and using \cref{prop:rank-dense-set}.
Take any $w \in \mathcal{E}$.%
The Segre variety $\Sigma_w^A$ is an $n$-dimensional subspace,
and $\Sigma_{F(w)}^{B}$ and $\Sigma_{F(Sw)}^{J_N}$ are $N$-dimensional
subspaces.  In other words, we find that
\begin{equation}
\Sigma_{F(Sw)}^{J_N} = \Sigma_{F(w)}^B .
\end{equation}
That is, $\inner{J_N Z,F(Sw)} = 0$ and  
$\inner{J_N Z, (J_N^*)^{-1} B^* F(w)} = \inner{B Z,F(w)} = 0$ denote the same subspace.
Hence their normal lines agree, that is,
for every $w \in \mathcal{E}$,
$F(Sw)$ is proportional to $TF(w)$,
where $T = (J_N^*)^{-1} B^* = J_N B$.  In other words,
\begin{equation}
\label{eq:proportional-vectors}
    F(Sw) = \lambda(w) TF(w)
\end{equation}
for some function $\lambda$.  The function $\lambda$ is defined on the open dense set $\mathcal{E}$,
and hence it is a rational function, $\lambda(w) = \frac{F_j(Sw)}{(TF)_j(w)}$ for any $j=0,\ldots,N$.  The equality
\eqref{eq:proportional-vectors} then holds outside the pole set of $\lambda$
by the identity theorem.  As the functions $F_0,\ldots,F_N$ have
no common factor, we see that the zero set of $\lambda(w)$ is
of codimension larger than 1, and as $\lambda$ is a single rational function
that means that $\lambda$ has no zeros and hence $\lambda$ is constant,
and we find that $F(Sw)=\lambda TF(w)$ for all $w$.
%
%, or equivalently,
%every $2 \times 2$ minor
%eq\begin{equation}
  %eq  P_{ij} = F_i(Sw) (TF)_j(w) - F_j(Sw) (TF)_i(w), \quad 0 \leq i, j \leq N
%\end{equation}
%vanishes.
%Since each polynomial $P_{ij}$ vanishes on the dense set $\mathcal{E}$,
%it vanishes identically, and hence 
%\cref{eq:proportional} is true 
%for every $w$ where $F(w) \neq 0$ and $FS(w) \neq 0$, that is, on an open and dense set.
%
If we take $\alpha$ to be the LFT given by the matrix $S$ and $\beta$
to be the
LFT given by the matrix $T$, then we obtain in affine coordinates that
\begin{equation}
f \circ \alpha = \beta \circ f
\end{equation}
as required.  The formulas for $\alpha$ and $\beta$ follow.

Finally, $\alpha$ is the identity of $\bP^n$ if and only if 
$J_n A = S = c I_{n+1}$
for some $c \in \C^*$, or equivalently, $A = c J_n$.
Comparing signatures gives us $c > 0$, that is,
$\inner{A z, z} = 0$ and $\inner{J_n z, z} = 0$ define the same source LFT sphere,
a contradiction.
A similar computation shows that $\beta$ is not the identity of $\bP^N$
as the two target LFT spheres are also distinct.
\end{proof}

In particular, $f$ now also takes the LFT sphere $\alpha^{-1}(S^{2n-1})$ to 
the LFT sphere $\beta^{-1}(S^{2N-1})$.
It turns out that these spheres need
not be the ones we started with.
By iterating, we find that
for any $k \in \Z$,
\begin{equation}
f \circ \alpha^k = \beta^k \circ f .
\end{equation}
So we also get that 
$f$ takes
$\alpha^{-k}(S^{2n-1})$ to 
$\beta^{-k}(S^{2N-1})$ for any $k$.
In our particular case, when
all spheres involved are centered at zero, $f$ takes
infinitely many spheres centered at zero
of various radii to spheres centered at zero.
In particular, in the first step, whenever it takes the standard sphere
to the standard sphere and the $r$-sphere to the $R$-sphere, 
we find that it also takes the $r^2$-sphere to the $R^2$-sphere and so on.

\begin{lemma} \label[lemma]{lemma:Nm1spherestospheres}
Suppose $n \geq 2$ and $f \colon \C^n \dashrightarrow \C^N$ is a
rational map such that $f(S^{2n-1}) \subset S^{2N-1}$ and for
some $r, R < 1$ we have that 
$f(rS^{2n-1}) \subset R S^{2N-1}$.
Suppose also that $k_f = N-1$.
Then there are two sequences of distinct numbers
$\{r_k\}_{k=1}^\infty$ and $\{ R_k \}_{k=1}^\infty$ converging to zero
so that 
$f(r_kS^{2n-1}) \subset R_k S^{2N-1}$ for every $k$.
\end{lemma}

\begin{proof}
We set things up as in the proof of \cref{thm:lft}, and so the matrices $A$ and $B$
are given as
\begin{equation}
A =
\begin{bmatrix}
-r^2    & 0      & \cdots & 0      \\
 0      & 1      & \cdots & 0      \\
 \vdots & \vdots & \ddots & \vdots \\
 0      & 0      & \cdots & 1      \\
\end{bmatrix}
\quad\text{and}\quad
B =
\begin{bmatrix}
-R^2    & 0      & \cdots & 0      \\
 0      & 1      & \cdots & 0      \\
 \vdots & \vdots & \ddots & \vdots \\
 0      & 0      & \cdots & 1      \\
\end{bmatrix} .
\end{equation}
Then
\begin{equation}
S = J_n A = 
\begin{bmatrix}
 r^2    & 0      & \cdots & 0      \\
 0      & 1      & \cdots & 0      \\
 \vdots & \vdots & \ddots & \vdots \\
 0      & 0      & \cdots & 1      \\
\end{bmatrix}
\quad
\text{and}
\quad
T = J_N B = 
\begin{bmatrix}
 R^2    & 0      & \cdots & 0      \\
 0      & 1      & \cdots & 0      \\
 \vdots & \vdots & \ddots & \vdots \\
 0      & 0      & \cdots & 1      \\
\end{bmatrix} .
\end{equation}
As we said what we have is that $f$ takes $\alpha^{-1}(S^{2n-1})$ to
$\beta^{-1}(S^{2N-1})$.  In affine coordinates,
$\alpha$ is simply a scaling by $\frac{1}{r^2}$ and 
$\beta$ is a scaling by $\frac{1}{R^2}$, and so
$\alpha^{-1}(S^{2n-1}) = r^2 S^{2n-1}$ and 
$\beta^{-1}(S^{2N-1}) = R^2 S^{2N-1}$.  By iterating we get
the desired sequence.
\end{proof}

Both \cref{thm:kfN1impliesHd} and
\cref{thm:kfN1impliesHdv2} will follows from the following lemma
using our previous work in \cite{helal-2025-proper}.

\begin{lemma} \label{lemma:kfN1infspheresrationaltoHd}
Suppose $n \geq 2$ and $f \colon \C^n \dashrightarrow \C^N$ is a
rational map.  Suppose
there are two sequences of distinct numbers
$\{r_k\}_{k=1}^\infty$ and $\{ R_k \}_{k=1}^\infty$ converging to zero
so that 
$f(r_kS^{2n-1}) \subset R_k S^{2N-1}$ for every $k$.
Suppose also that $k_f = N-1$.
Then there exists a unitary map $U$ so that $f = UH_d$.
\end{lemma}

\begin{proof}
In \cite{helal-2025-proper}, we have proved that a rational map
that takes infinitely many spheres centered at zero to 
spheres centered at zero
must, up to a postcomposition with a unitary,
be a direct sum of scaled maps $H_d$ for various $d$ or zeros.
That is, for some unitary $U$,
\begin{equation}
f = U(c_1 H_{d_1} \oplus \cdots \oplus c_\ell H_{d_{\ell}} \oplus 0 ),
\end{equation}
where we can assume that $d_1, \ldots, d_\ell$ are distinct.
Since $k_f = N-1$, we must have that $N_f = N$
and hence there are no zero components.
Then, via the argument in the proof of \cref{prop:kfHd},
we can see that $k_f \leq N-\ell$.  Hence $\ell = 1$,
and $f = U H_d$ as required.
\end{proof}

We are now ready to prove \cref{thm:kfN1impliesHd} and
\cref{thm:kfN1impliesHdv2}.

\begin{proof}[Proof of \cref{thm:kfN1impliesHd}]
One direction of the theorem is clear and follows from the discussion so far.
What is left to prove is that if $f \colon \A_{n,r} \to \A_{N,R}$ is a
proper holomorphic map and $k_f = N-1$, then $f = U H_d$ for some unitary
$U$ and an integer $d$.  By \cref{theorem:annulimapsrational}, $f$ is
rational and takes spheres to spheres as in
\cref{lemma:Nm1spherestospheres}.
Therefore, $f$ takes infinitely many
spheres centered at zero to spheres centered at zero.
The conclusion follows by \cref{lemma:kfN1infspheresrationaltoHd}.
\end{proof}

The proof of \cref{thm:kfN1impliesHdv2} follows by the same logic.

\begin{proof}[Proof of \cref{thm:kfN1impliesHdv2}]
If $f$ is a holomorphic map on a neighborhood of the smaller ball $r\bB_n$
and takes $rS^{2n-1}$ to $RS^{2N-1}$.
By the maximum principle, $f$ takes $r\bB_n$ to $R\bB_N$
and as it takes boundary to boundary, it is a proper map.
By the theorem of Forsteri\v{c}, $f$ is rational and takes
$S^{2n-1}$ to $S^{2N-1}$.
By \cref{lemma:Nm1spherestospheres}, $f$ takes infinitely many spheres
centered at zero to infinitely many spheres centered at zero,
and the proof follows, again, by \cref{lemma:kfN1infspheresrationaltoHd}.
\end{proof}

Let us now prove \cref{theorem:farananalogue}
as an application of
\cref{thm:kfN1impliesHd}.
We already have that proper holomorphic maps from $\A_{2,r}$ to $\A_{3,R}$
are rational.  If the degree of the map is 1, then
via \cref{prop:linearembeddings},
it is
unitarily equivalent to the affine embedding
\begin{equation}
(z_1,z_2) \mapsto
\left(
\sqrt{\frac{1-R^2}{1-r^2}} \, z_1 ,
\sqrt{\frac{1-R^2}{1-r^2}} \, z_2 ,
\sqrt{\frac{R^2-r^2}{1-r^2}} \right),
\end{equation}
where $R \geq r$.  In this case, clearly $N_f=2$ and $k_f=1$.

The theorem will therefore follow if we can show that any map
of degree greater than 1 is unitarily equivalent to $H_2$.
This result follows by
Faran's~\cite{faran-1982-maps} classification of
proper maps of $\B_2$ to $\B_3$.  The following lemma finishes the
proof of \cref{theorem:farananalogue}.

\begin{lemma}
Suppose that $f \colon \A_{2,r} \to \A_{3,R}$ is a rational
proper map of degree greater than 1.  Then $f$ is unitarily
equivalent to the homogeneous map $H_2$.
\end{lemma}

\begin{proof}
The map $f$ extends to a proper map of $\B_2$ to $\B_3$.
Faran's result implies that a proper rational map of $\B_2$ to $\B_3$
is spherically equivalent to one of four maps:
the linear embedding $(z_1,z_2,0)$, the Whitney map
$(z_1,z_1z_2,z_2^2)$, the homogeneous $(z_1^2,\sqrt{2}\,z_1z_2,z_2^2)$,
and the map $(z_1^3,\sqrt{3}\, z_1z_2, z_2^3)$.
The issue is that the automorphisms used in Faran's classification
are automorphisms of the ball and these do not preserve the
inner sphere boundary of the annulus.
Spherical equivalence does, however, preserve the degree of the map
and the integer $k_f$.
For the three maps of degree greater than 1, we find $k_f = 2 = N-1$.
To compute $k_f$,
it is sufficient to consider the hyperplanes
$z_1 = \text{constant}$.
In other words, if we have an annulus map as in the statement
and of degree greater than 1, then $k_f = 2 = N-1$.
Hence
\cref{thm:kfN1impliesHd} says
that $f$ is unitarily equivalent to $H_2$.
\end{proof}

% ==============================================================

\section{Complete normal form of degree 2 annulus maps} \label[section]{sec:seconddeg}

In the following theorem, it will be useful to write $H'_2(z_2,\ldots,z_n)$,
that is, the symmetrized homogeneous map in the variables $z_2,\ldots,z_n$:
\begin{equation}
H'_2(z_2,\ldots,z_n) =
\bigl(
z_2^2,
\sqrt{2} z_2z_3, \ldots, \sqrt{2} z_2z_n,
z_3^2,
\sqrt{2} z_3z_4, \ldots, \sqrt{2} z_3z_n,
\ldots,
z_n^2 \bigr) .
\end{equation}
Note that when $n=2$, which is the main case to keep in mind, $H'_2(z_2) =
z_2^2$.

The complete classification of rational degree 2 maps of annuli is then the
following theorem.  One should contrast this with the much simpler
classification for quadratic proper maps of balls, where up to spherical
equivalence, every such map is monomial.  All annuli maps are also ball
maps, and hence all of the following maps are in fact spherically equivalent
to monomial maps, but here we have to classify up to the unitary group.
On the other hand, compared to the ball case,
thanks to \cref{lemma:quotientlemma}, the classification of annuli degree 2 maps follows by
quite tedious but elementary methods, not requiring sophisticated results in linear algebra.  In order to simplify the statement we assume that $N=N_f$.  Therefore, to get all degree two maps in general one would also have to allow compositions with linear embeddings.

\begin{theorem}
\label[theorem]{thm:seconddeg}
Suppose $f \colon \A_{n,r} \to \A_{N,R}$ is a proper rational map of degree 2,
$n \geq 2$, such that $N_f = N$, and write
$b=\frac{1-R^2}{1-r^2}$.  Then $R \geq r^2$ and:

\begin{enumerate}
\item
If $R = r^2$, then $N = \binom{n+1}{2}$ and $f$ is unitarily equivalent to $H_2$.
\item
If $r^2 < R < r$ and $f$ is polynomial, then $N=\binom{n+1}{2}+n$ and
$f$ is unitarily equivalent to the juxtaposition
$\sqrt{1-t} \, H_1 \oplus \sqrt{t} \, H_2$
for exactly one $t \in (0,1)$.
\item
If $r^2 < R < r$ and $f$ is not polynomial, then $N=\binom{n+1}{2}+n$ and
$f$ is unitarily equivalent to a map of the form
\begin{equation}
\label{eq:generaldeg2mapform}
\begin{split}
\Biggl( &
\frac{\sqrt{1-b+Qr^2} + \frac{(1-b)a}{\sqrt{1-b+Qr^2}}z_1}{1+az_1}
,
\frac{\sqrt{ba^2+Q}\,z_1^2 + \frac{ba}{\sqrt{ba^2+Q}}z_1}{1+az_1}
,
\\
& \qquad
\frac{\sqrt{b-Q(1+r^2)} + \frac{ba}{\sqrt{b-Q(1+r^2)}} z_1}{1+az_1}
\otimes (z_2,\ldots,z_n)
,
\\
& \qquad
\frac{\sqrt{ba^2+2Q-\frac{b^2a^2}{b-Q(1+r^2)}} \, z_1}{1+az_1}
\otimes (z_2,\ldots,z_n)
,
\frac{\sqrt{Q}}{1+az_1} H'_2(z_2,\ldots,z_n)
\Biggr)
\end{split}
\end{equation}
for exactly one pair $a$ and $Q$, where $a > 0$,
\begin{equation}
\label{eq:upperbounda}
a \leq
\sqrt{
\frac{
br^4+(b^2-b+1)r^2-b+1
-2r\sqrt{(b-1)b(1+r^2)(br^2-b+1)}
}{
b^2r^6+2br^4+r^2
}
}
\end{equation}
and $Q$ is either one of
\begin{multline}
\label{eq:solutionsQ}
Q=
\frac{
-a^2br^4+(2(1-a^2)b+a^2-1)r^2+b-1
}{
2(r^4+r^2)
}
\\
\pm
\frac{
\sqrt{
a^4r^4(br^2+1)^2
-2a^2r^2(br^2(r^2+b-1)+1+r^2-b)
+(1+r^2-b)^2
}
}{
2(r^4+r^2)
}
.
\end{multline}
%\begin{equation}
%\label{eq:solutionsQ}
%Q=
%\frac{
%\splitdfrac{
%\pm
%\sqrt{
%a^4r^4(br^2+1)^2
%-2a^2r^2(br^2(r^2+b-1)+1+r^2-b)
%+(1+r^2-b)^2
%}
%}{
%-a^2br^4+(2(1-a^2)b+a^2-1)r^2+b-1
%}
%}{
%2(r^4+r^2)
%}
%.
%\end{equation}
Moreover, for each such pair $a$ and $Q$, a map exists.
\item
If $R = r$, then $N=\binom{n+1}{2}+n$ and
$f$ is unitarily equivalent to a map of the form
\eqref{eq:generaldeg2mapform}
%\begin{multline}
%\label{eq:b1deg2mapform}
%\Biggl(
%\frac{r \sqrt{\frac{1}{1+r^2}-a^2}}{1+az_1}
%,
%\frac{a \sqrt{1+r^2}z_1+\frac{1}{\sqrt{1+r^2}}z_1^2}{1+az_1}
%,
%\frac{a \sqrt{1+r^2}+\frac{1}{\sqrt{1+r^2}}z_1}{1+az_1}
%\otimes (z_2,\ldots,z_n)
%,
%\\
%\frac{\sqrt{\frac{1}{1+r^2}-a^2} \, z_1 }{1+az_1}
%\otimes (z_2,\ldots,z_n)
%,
%\frac{\sqrt{\frac{1}{1+r^2}-a^2} }{1+az_1} H'_2(z_2,\ldots,z_n)
%\Biggr)
%\end{multline}
for exactly one $a \in \left(0,\frac{1}{\sqrt{1+r^2}}\right)$
and $Q=\frac{1}{1+r^2}-a^2$.
Moreover, for each such $a$, a map exists.
\item
If $r < R < 1$, then $N=\binom{n+1}{2}+n$ and 
$f$ is unitarily equivalent to a map of the form
\eqref{eq:generaldeg2mapform}
for exactly one $a \in (0,1)$ where $Q$ is the larger
value of \eqref{eq:solutionsQ}.
Moreover, for each such $a$, a map exists.
\end{enumerate}
\end{theorem}

We observe that by part (v) we obtain maps where
the embedding dimension equals the target dimension and $R$ is arbitrarily
close to $1$.
It is not possible to do the same with
juxtapositions of homogeneous maps only.

The key part of the proof is to use \cref{lemma:quotientlemma}.  In the case
of degree 2 maps we quickly obtain the following first
step in the classification.

\begin{proposition}
Suppose $f = \frac{p}{q} \colon \A_{n,r} \to \A_{N,R}$ is a proper rational map of degree 2,
$n \geq 2$, take $q(0)=1$, and write $b=\frac{1-R^2}{1-r^2}$.  Then after a possible unitary
transformation on the source,
\begin{equation}
\label{eq:deg2form}
\norm{p(z)}^2
= \Bigl(1 + b \bigl(\norm{z}^2-1\bigr)\Bigr) \abs{1+az_1}^2 
+ Q \bigl(\norm{z}^2-1\bigr) \bigl(\norm{z}^2-r^2\bigr) ,
\end{equation}
where $0 \leq a < 1$ and $Q > 0$ are constants and these two
constants are the complete invariants for such maps under unitary equivalence.
\end{proposition}

\begin{proof}
We apply \cref{lemma:quotientlemma}, and as the degree is 2, we find that
$Q$ is constant and $q$ is
of degree 1 or 0.  We can take $q(0)=1$ and so, $q(z) = 1+ v \cdot z$
for some vector $v \in \C^n$.
If $v=0$, we are in the polynomial case and $q \equiv 1$.
Note that pre- or postcomposing with a unitary cannot change
whether the map $f$ is polynomial or rational.
In the rational case, that is, $v \not=0$,
we precompose with a unitary so that $v \cdot z = a z_1$ for some $a > 0$.
In the polynomial case, $v = 0$ and hence $v \cdot z = a z_1$ for $a=0$.
In either case, precomposing $f$ (and hence
$\norm{p(z)}^2$) with a unitary cannot change $a$ or $Q$ after this
normalization.
Postcomposing $f$ with a
unitary is the same as postcomposing $p$ with a unitary and this leaves
$\norm{p(z)}^2$ unchanged.  In other words, the normalized
polynomial $\norm{p(z)}^2$ is a complete invariant under unitary equivalence.
It is not difficult to
see that $\norm{p(z)}^2$ uniquely determines $Q$ and $a$ and hence they are
complete invariants under unitary equivalence.

As $f$ is a rational ball map, we know that it extends past the boundary and
hence $q$ cannot be zero on the closed unit ball, which means that $a < 1$.
By considering the second degree terms in $p$, that is,
$\norm{p_2(z)}^2 = b \norm{z}^2 \abs{a z_1}^2 + Q \norm{z}^4$, we find that
$Q \geq 0$ by setting $z_1=0$.  Moreover,
if $Q=0$, then every component in $p$ would be divisible by $q$, meaning
that $f$ is not in fact a degree 2 map but a degree 1 map.  So $Q > 0$.
\end{proof}

Therefore to complete the proof of the theorem, we simply need to find
conditions on $b$, $r$, $a$, and $Q$ so that the expression on the
right hand side of \cref{eq:deg2form} is a hermitian sum of squares.
Writing $\norm{p(z)}^2$ as a sum of hermitian squares is the same
as writing the hermitian matrix of coefficients as a sum of outer products
of vectors.  These vectors are the coefficients of the polynomial components
of $p(z)$.  In particular,
a real polynomial is a sum of hermitian squares if and only if
its hermitian matrix of coefficients is a positive semidefinite matrix.
If we are starting with the right hand side of \cref{eq:deg2form}
we can ensure that these vectors are linearly independent, which is
equivalent to $f$ not mapping into a subspace.  If the
components of the map are linearly independent, the rank of this matrix
of coefficients is precisely the target dimension $N$ of the map $f$.

If the map is polynomial, that is $a=0$, then we already
know from \cite{helal-2025-proper} that $f$ is unitarily equivalent
to a juxtaposition of homogeneous maps.  Since we are assuming $N=N_f$
in the theorem, none of these maps can be the constant.   Thus since the
map must be degree 2, it must be either equivalent to $H_2$ or
a juxtaposition $\sqrt{1-t} H_1 \oplus \sqrt{t} H_2$
for $t \in (0,1)$.  We have therefore finished items (i) and (ii)
of the theorem.
In what follows, we will mostly consider the rational case $a > 0$.

Next we also notice that it is sufficient to consider the case $n=2$.  This
is because the right hand side of \cref{eq:deg2form} is an
expression that is a function of $z_1$, $\bar{z}_1$ and
$\abs{z_2}^2+\cdots+\abs{z_n}^2$.  Therefore, if we find the normal form in
the case $n=2$, the larger dimensions will follow by replacing
$\abs{z_2}^2$ with 
$\abs{z_2}^2+\cdots+\abs{z_n}^2$ in the expression for $\norm{p(z)}^2$.

Let us then fix $n=2$ and write the matrix of coefficients of
the right hand side of \cref{eq:deg2form}.  We order the monomials
as $1,z_1,z_1^2,z_2,z_1z_2,z_2^2$.  With this ordering the matrix becomes
\begin{equation}
\begin{bmatrix}
(1-b) + Qr^2 & (1-b)a & 0 & 0 & 0 & 0 \\
(1-b)a & (1-b)a^2+b-Q(1+r^2) & ba & 0 & 0 & 0 \\
0 & ba & ba^2+Q & 0 & 0 & 0 \\
0 & 0 & 0 & b-Q(1+r^2) & ba & 0 \\
0 & 0 & 0 & ba & ba^2+2Q & 0 \\
0 & 0 & 0 & 0 & 0 & Q
\end{bmatrix} .
\end{equation}
That is, the $(k,\ell)$-th entry of this matrix is the coefficient
of the polynomial corresponding to the $\ell$th monomial in the ordering
times the conjugate of the $k$th monomial.

Call this coefficient matrix $M$.  Notice that we picked an ordering
of the monomials so that the matrix is a direct sum of three matrices,
$M = M_1 \oplus M_2 \oplus M_3$ for a $3 \times 3$ $M_1$, a $2 \times 2$
$M_2$, and $M_3 = [Q]$.  The three matrices $M_1$, $M_2$, and $M_3$
must be positive semidefinite.

Before we move on, we notice that given a
right hand side of \cref{eq:deg2form} for $n=2$, we can find
a $p(z)$ where all components of $p(z)$ that depend on $z_2$ are either
of the form $L(z_1) z_2$ for an affine $L(z_1)$, or the last
term $\sqrt{Q} \, z_2^2$.  That means that if we find a presentation
of a representative of the equivalence class of maps for $n=2$,
to get the corresponding map for $n > 2$ we simply replace
the multiplication by $z_2$ with tensoring by $(z_2,\ldots,z_n)$
and we replace $z_2^2$ by the symmetrized second tensor power of
$(z_2,\ldots,z_n)$, that is, by $H'_2(z_2,\ldots,z_n)$.  That is
where those terms in \cref{eq:generaldeg2mapform} come from.
From now on we therefore generally assume that $n=2$.

The block $M_3$ is trivial as $Q > 0$ and gives rise to
the term $\sqrt{Q}\, z_2^2$.  We will show that the matrix $M_2$
is always positive definite (nonsingular in fact) when $M_1$ is positive
semidefinite, and so it will not be an issue.

\begin{proposition}
For a degree 2 annulus map $f$, we have $Q \leq 1$ and $0 < b \leq 1+r^2$.
In particular, we have that $R \geq r^2$.
Furthermore, if $f$ is not polynomial, then
$Q < 1$ and $b < 1+r^2$ (so $R > r^2$).
\end{proposition}

\begin{proof}
Clearly $b > 0$.
If we consider the first diagonal entries of $M_1$ and $M_2$, 
we find that they must be
nonnegative, that is,
$1-b + Qr^2 \geq 0$ and 
$b- Q(1+r^2) \geq 0$.  Together this means that
$Q(1+r^2) \leq b \leq 1 + Qr^2$.  This implies that $Q \leq 1$
and $b \leq 1+r^2$.
Suppose now that $Q=1$.  Then $b=1+r^2$ and also the first diagonal entry
of $M_2$ is now $0$.  Thus for the matrix to be positive semidefinite
we require the off-diagonal entry $ba = 0$.  Since $b > 0$, we find that
$a=0$.  Thus if $a > 0$, we have that $Q < 1$.  Consequently $b < 1+r^2$.
\end{proof}

For the rest of the argument we will mostly consider the matrix
\begin{equation}
M_1 = 
\begin{bmatrix}
(1-b) + Qr^2 & (1-b)a & 0 \\
(1-b)a & (1-b)a^2+b-Q(1+r^2) & ba \\
0 & ba & ba^2+Q 
\end{bmatrix} .
\end{equation}
By setting $z_2=0$, we have that $M_1$ is a matrix of coefficients
corresponding to a proper annulus map with source dimension 1.
Assuming that $f$ does not map into a subspace, the matrix $M_1$ cannot
be zero (as not all components of $p$ can be divisible by $z_2$).
If the rank of $M_1$ is 1, then we find an annulus map from dimension 1
to dimension 1.  We also know that the numerator $p$ has to be
of degree $2$.  This map therefore has to be $e^{i\theta} z^2$, meaning
$Q=1$, $a=0$, and $b=1+r^2$.  As we are assuming that $a > 0$, this is not
a possibility.  Therefore, $M_1$ must be of rank 2 or 3.

\begin{proposition}
Suppose $p$ has linearly independent components and $a > 0$.
The target dimension $N$ of $f$ is equal to the embedding dimension $N_f$
if and only if $M_1$ is of rank 2.
\end{proposition}

\begin{proof}
We are left with $M_1$ being rank 2 or 3.  If $M_1$ is of rank 3 and hence
of full rank, then
we can subtract from $M_1$ a multiple of the matrix $A$ corresponding to
$\abs{1+az_1}^2$ and $M_1 - \epsilon A$ is still positive (semi)definite
for small enough $\epsilon > 0$.
Pick $\epsilon$ large enough so that 
$M_1 - \epsilon A$ is rank 2, which we decompose as a sum of outer products
as usual.  Thus we find a decomposition $\norm{p(z)}^2$ with linearly
independent components, where one of the components of $p$ is
$\sqrt{\epsilon}(1+az_1)$.  But that means that the corresponding
$f=\frac{p}{q}$ has a constant term, meaning that the embedding dimension
for $f$ is strictly smaller than $N$.

Conversely, suppose $M_1$ is rank 2 and the embedding dimension is strictly
smaller than $N$.  This means that after a unitary, we can make one of the
components of $f$ be a constant, meaning that $\norm{f(z_1,0,\ldots,0)}^2=\abs{f_1(z_1)}^2+c^2$, as only the $M_1$
block contributes terms and the $\abs{f_1(z_1)}^2$ comes from the second
eigenvalue of $M_1$.  Again, this means that $f_1$ is a proper map of annuli
in the complex plane and hence just a multiple of $z^2$ and this leads to a
contradiction if $a > 0$ as before.
\end{proof}

\pagebreak[2]
As we are assuming that $M_1$ is singular, we have the equation
\begin{multline}
0 = \det(M_1) =
-r^2(r^2+1)Q^3
\\
+
(-a^2br^4-2a^2br^2+2br^2+a^2r^2-r^2+b-1)Q^2
+
(1-a)(1+a)(1-b)b(1-ar)(ar+1)Q .
\end{multline}
This equation is divisible by $Q$.  The root at $Q=0$ is not interesting
as we are interested in $Q > 0$, so we can discard it and we are left
with a quadratic equation.  The two solutions to this quadratic equation are
written out in \cref{eq:solutionsQ}.
As a cubic in $Q$, $\det(M_1)$ has negative cubic term.  It therefore
goes to negative infinity as $Q \to \infty$.

It is now good to consider the cases $b < 1$ ($R > r$), $b=1$ ($R=r$), and
$b>1$ ($R < r$) separately.
We start with $b < 1$.  In this case, $1-b > 0$.  We note that
\begin{equation}
\left.\frac{\partial}{\partial Q}\right|_{Q=0}\det(M_1) =
(1-a)(1+a)(1-b)b(1-ar)(ar+1) > 0.
\end{equation}
This means that as a function of $Q$, $\det(M_1)$ must have a negative root
$Q_-$,
the aforementioned root at $0$, and a positive root $Q_{+}$.
By looking at the matrix $M_1$ when $a=0$, we find that it has 3 positive
eigenvalues when $Q > 0$ is small enough.  As the sign of an eigenvalue
can only change when the determinant goes through 0, we have that
$M_1$ is positive semidefinite for all $0 \leq Q \leq Q_{+}$.
For similar continuity reasons the same is true for all $0 < a < 1$.
As $\det(M_1)$ is negative for all $Q > Q_{+}$, the matrix $M_1$
cannot be positive semidefinite for $Q > Q_{+}$.

Next we still need to handle what happens at $Q=1$.  While we have shown
that for a map to exist, we must have $Q < 1$, what we want to show is that
$Q_+ < 1$ and hence a map always exists.  For this we plug in $Q=1$ into
$\det(M_1)$.
\begin{equation}
\det(M_1)|_{Q=1} =
-(a^2b+1)r^4
-(a^4b^2-a^2b^2-a^4b+3a^2b-2b-a^2+2)r^2
-(1-b)(a^2b+1-b) .
\end{equation}
That is, we have a polynomial in $r$.  It is clear that for the considered
values of $a \in (0,1)$ and $b \in (0,1)$ the coefficient
of $r^4$ and the constant are strictly negative.  The coefficient
of $r^2$ is also negative in this range which is a simple calculus exercise
as the expression has a strict minimum on the unit square at $a=0,b=1$.
By Descartes' rule of signs, there are no positive solutions $r$
for $\det(M_1)|_{Q=1} = 0$.  Therefore $\det(M_1)|_{Q=1}$ is always strictly
negative for the allowed $a$, $b$, and $r$, and hence $Q_+ < 1$.

We need to next show that for the range of $Q$ where $\det(M_1)$ is
nonnegative, we have that $\det(M_2) > 0$.  We have
\begin{equation}
\det(M_2) = (-2(r^2+1)Q + (2-a^2r^2-a^2)b) Q .
\end{equation}
This function has two roots, one at the origin and one positive, call it
$Q_{2,+}$.
It is easy to see that at $a=0$ the matrix is positive semidefinite
and since $\left.\frac{\partial}{\partial Q}\right|_{Q=0}\det(M_2) > 0$
we find that $M_2$ is positive semidefinite between the two roots.
We plug $Q=Q_{2,+}$ into $\det(M_1)$ to find
\begin{equation}
\frac{
a^2b^2(2-a^2r^2-a^2)\bigl((a^2r^6-a^2r^2-2r^2+2)b+(2a^2r^4-2r^4+2a^2r^2-2)\bigr)
}{
8(r^2+1)^2
} .
\end{equation}
The first two terms of the numerator are clearly strictly positive in our allowed
range.  It is also not hard to see by basic calculus that the third term of
the numerator is strictly negative (it grows as $b$ grows and is still negative at
$b=1$).

In other words, we have shown that at $Q=Q_+$, $M_2$ is still positive
definite.
 Hence, we have a coefficient matrix $M$ that is positive semidefinite
and writing the blocks of $M$ as sums of hermitian squares obtains the form
in the theorem. 
As $M_2$ is nonsingular and $M_1$ is rank 2,
it is not hard to check that the matrix $M$ is of full rank minus 1 for
any $n \geq 2$.
Hence we are finished with $b < 1$ and the item (v)
of the theorem.

Let us now consider $b=1$.  When $b=1$, we find that $\det(M_1)$ simplifies
greatly to
\begin{equation}
\bigl(r^2(1-a^2(1+r^2))-Qr^2(1+r^2)\bigr) Q^2 .
\end{equation}
There are two roots at $Q=0$ and one other root $Q_0 = \frac{1}{1+r^2} - a^2$.
If the root is negative then the determinant is negative for positive $Q$,
hence we only have a map if the root is positive.  It is positive
if and only if $\frac{1}{\sqrt{1+r^2}} > a$, which gives the restriction on
$a$ in item (iv) of the theorem.  It is also clear that $Q_0 < 1$.
Moreover, plugging $Q=Q_0$ into $\det(M_2)$ we obtain
\begin{equation}
a^2\bigl(1-a^2(1+r^2)\bigr) .
\end{equation}
Thus $\det(M_2)$ is strictly positive as are the two diagonal terms, and so 
the matrix $M_2$ is positive definite.  Item (iv) of the theorem now follows
by expanding $M$ in the same way as before.

Finally, we consider $1 < b < 1+r^2$.
We first examine the case when $a=0$.  In this case the matrix
$M_1$ is diagonal and the roots of 
$\det(M_1)=0$ are $Q=0$,
$Q_1 = \frac{b-1}{r^2}$ and $Q_2 = \frac{b}{1+r^2}$, which satisfy $0 < Q_1
< Q_2 < 1$ by the assumption on $1 < b < 1+r^2$.
For any $Q \in (Q_1,Q_2)$, the diagonal entries of $M_1$ are strictly
positive so $M_1$ is positive definite.
Let us now track how the roots $Q_1$ and $Q_2$ change as $a$ increases from $0$.
The roots depend continuously on $a$.
The coefficient of $Q$ in $\det(M_1)$
is $(1-a)(1+a)(1-b)b(1-ar)(ar+1)$ and that is nonzero when
$a \in [0,1)$ and $b > 1$ and so we see that the root at $Q=0$ is
always a simple root.
Therefore the other two roots cannot cross zero as $a$ moves.

It is not hard to see that $M_1|_{Q=1}$
is not positive semidefinite.
Otherwise the center term would be nonnegative: $(1-b)a^2+b-(1+r^2) \geq 0$.
Together with $b < 1+r^2$,
we would obtain $(1-b)a^2 > 0$,
which would contradict $1 < b$.
Therefore, as $a$ grows the two roots must stay less than $1$ until we
either reach $a=1$ or they
possibly disappear (become complex) if the discriminant becomes zero.  In this range,
as the matrix $M_1$ is positive semidefinite for $Q$ between the two roots at $a=0$
it is also positive semidefinite throughout.
To find if and where the two roots disappear, we consider the
discriminant of the quadratic $\det(M_1)/Q=0$:
\begin{equation} \label{eq:discriminantM1byQ}
a^4r^4(br^2+1)^2
-2a^2r^2(br^2(r^2+b-1)+1+r^2-b)
+(1+r^2-b)^2 .
\end{equation}
This expression is strictly positive when $a=0$ and at $a=1$ it is
$(1-b(1-r^4))^2 \geq 0$.  The expression can be considered a quadratic
in $a^2$ and it is concave up.  To show that it always has two roots between $0$ and $1$,
it is better to substitute $b=1+sr^2$ so that we have a parameter $s \in
(0,1)$ just like $a$ and $r$.  The minimum of this polynomial is
at
\begin{equation}
a^2 = \frac{r^4s^2+r^4s+r^2s-s+r^2+1}{(r^4s+r^2+1)^2} .
\end{equation}
It is a calculus exercise that for $s,r$ in $(0,1)$, the $a$ where
this minimum occurs is always in $(0,1)$.
We plug the location of this minimum into the discriminant
\cref{eq:discriminantM1byQ} (with $b=1+sr^2$) to obtain
\begin{equation}
\frac{-4r^6(r^2+1)s(r^2s+1)(r^2s-s+1)}{(r^4s+r^2+1)^2} .
\end{equation}
It is not hard to see that this minimum is then always negative.
Putting this all together means that the discriminant
\cref{eq:discriminantM1byQ} always has two positive roots
(as a function of $a^2$),
both of these roots are within $(0,1]$, with the smaller one
strictly less than $1$.
Let $a_1$ and $a_2$ denote the two roots where $a_1$ is the smaller one.
The smaller root is given as the right
hand side in \cref{eq:upperbounda}. For $0 < a \leq a_1$,
we therefore obtain two roots $Q_1$ and $Q_2$ in $(0,1)$ to 
$\det(M_1)=0$ and between them the matrix $M_1$ is positive semidefinite,
which can be checked by plugging in a test point between the two roots and noting that by 
continuity the entire connected region given by
$Q_1 < Q < Q_2$, $0 < a < a_1$, $s \in (0,1)$, where
the determinant is strictly positive, we must get that
the matrix is positive definite.  Thus it is semidefinite when we allow
nonstrict inequalities.  Similarly we check the regions
$Q_1 < Q < Q_2$, $a_1 < a < 1$, $s \in (0,1)$ (there are two such regions
as $a_1 = 1$ if $b=\frac{1}{1-r^4}$ as we have seen above).
In both of these regions the matrix $M_1$ has a negative eigenvalue.

We are done with the matrix $M_1$.  It remains to show that $M_2$ is
positive definite if the parameters are within the stated limits.
This is not hard to see, if we delete the first row and first column of
$M_1$ we get a positive semidefinite matrix $M_1'$.  Then $M_2$ can be
obtained from $M_1'$ by adding strictly positive values to both diagonal
entries, hence $M_2$ is in fact positive definite.
Item number (iii) now follows.

% ==============================================================

\section{Other remarks on proper maps of annuli} \label[section]{sec:remarks}

Let us put together some further minor remarks.
Firstly, there are two other domains in $\C^n$ whose automorphism group is
$U(n)$.  There is the punctured ball $\B_n \setminus \{ 0 \}$ and
the complement of the closed unit ball $\C^n \setminus \overline{\B_n}$.
For maps from a punctured ball to a punctured ball, the maps must extend
through the origin by the Riemann extension theorem if $n=1$ or via the
Hartogs phenomenon if $n > 1$.  We are then simply dealing with proper maps
of balls that take the origin to the origin.  
Since the automorphism group of the ball is transitive,
every spherical equivalence class of proper maps of balls contains
a map taking $0$ to $0$.
The case of the complement of the closed unit ball was studied in
\cite{helal-2025-proper}.  In that case, for $n \geq 2$, the 
proper maps are polynomial sphere maps whose norm goes to infinity at
infinity.  In either case, the classification is different in these
degenerate cases, and in both cases there are more maps than for annuli.

Another remark we wish to make is that we have already seen that
in the first interval we have the restriction $r \leq R < 1$ for a map to
exist.  A question may arise about which $R$ admit a proper
map in general.
It is relatively simple to see that as long as we do not place any
restriction on $N$ we may get any $r$ and $R$ whatsoever.  That is, given
$n$, $r$, and $R$, pick $d$ so that $r^d \leq R$.  Then as $H_d$
takes the $r$-sphere to the $r^d$-sphere we have that
the map
\begin{equation}
z \mapsto \sqrt{\frac{1-R^2}{1-r^{2d}}} \, H_d(z) \oplus
\sqrt{\frac{R^2-r^{2d}}{1-r^{2d}}}
\end{equation}
is a proper map of $\A_{n,r}$ to $\A_{N,R}$
where $N = \binom{n+d-1}{d}+1$.

The previous discussion suggests that the existence of a rational proper
holomorphic map $f \colon \A_{n,r} \to \A_{N,R}$ of degree $d$ may imply a
relationship between $r$, $R$, and $d$.  For homogeneous maps
as above we have $r^d \leq R$, with $H_d$ itself achieving the equality
in this bound.  Juxtapositions (that is, weighted direct sums) of 
homogeneous maps, and of all the other maps the authors have constructed still
satisfy this bound, and so a reasonable question seems to be if 
$r^d \leq R$ (equivalently as a bound on degree, $d \geq \log_r R$)
holds in general.
In \cref{sec:seconddeg}, we show that all degree 2 maps satisfy
the bound $r^2 \leq R$, giving further support for this conjecture.
Moreover, we construct degree 2 maps with $r^2 < R < 1$
such that the target dimension $N$ is the embedding dimension.
That means that there is no simple upper bound on $R$ if the degree
$d$ is larger than 1.

%%%%%%%%%%%%%%%%%%%%%%%%%%%%%%%%%%%%%%%%%%%%%%%%%%%%%%%%%%%%%%%%%%%%%%%%%%%%%

\bibliographystyle{amsplain}
\bibliography{ball_maps}

@article{fatou-1923-fonctions,
	title        = {Sur les fonctions holomorphes et born\'ees \`a l'int\'erieur d'un cercle},
	author       = {Fatou, Pierre Joseph Louis},
	year         = 1923,
	journal      = {Bull. Soc. Math. France},
	fjournal     = {Bulletin de la Soci\'et\'e Math\'ematique de France},
	publisher    = {Soci\'et\'e math\'ematique de France},
	volume       = 51,
	pages        = {191--202},
	doi          = {10.24033/bsmf.1033},
	url          = {http://www.numdam.org/articles/10.24033/bsmf.1033/},
	language     = {fr},
	mrnumber     = 1504825,
	zbl          = {49.0221.01}
}

@article{alexander-1977-proper,
	title        = {Proper holomorphic mappings in {C$^n$}},
	author       = {Alexander, Herbert John},
	year         = 1977,
	journal      = {Indiana Univ. Math. J.},
	fjournal     = {Indiana University Mathematics Journal},
	publisher    = {Indiana University Mathematics Department},
	volume       = 26,
	number       = 1,
	pages        = {137--146},
	url          = {https://www.jstor.org/stable/24891328},
	mrnumber     = {0422699},
	zbl          = {0391.32015}
}

@article{webster-1979-mapping,
	title        = {On mapping an $n$-ball into an ($n+1$)-ball in complex spaces},
	author       = {Webster, Sidney Martin},
	year         = 1979,
	month        = mar,
	journal      = {Pacific J. Math.},
	fjournal     = {Pacific Journal of Mathematics},
	publisher    = {Mathematical Sciences Publishers},
	volume       = 81,
	number       = 1,
	pages        = {267--272},
	doi          = {10.2140/pjm.1979.81.267},
	url          = {https://msp.org/pjm/1979/81-1/p21.xhtml},
	mrnumber     = 543749,
	zbl          = {0379.32018}
}

@article{faran-1982-maps,
	title        = {Maps from the two-ball to the three-ball},
	author       = {Faran, James John},
	year         = 1982,
	month        = nov,
	journal      = {Invent. Math.},
	fjournal     = {Inventiones Mathematicae},
	publisher    = {Springer},
	volume       = 68,
	number       = 3,
	pages        = {441--475},
	doi          = {10.1007/bf01389412},
	url          = {https://link.springer.com/article/10.1007/BF01389412},
	mrnumber     = 669425,
	zbl          = {0519.32016}
}

@article{rudin-1984-homogeneous,
	title        = {Homogeneous polynomial maps},
	author       = {Rudin, Walter},
	year         = 1984,
	month        = mar,
	day          = 26,
	journal      = {Indagationes Mathematicae (Proceedings)},
	publisher    = {Elsevier B.V.},
	volume       = 87,
	number       = 1,
	pages        = {55--61},
	doi          = {10.1016/1385-7258(84)90057-x},
	issn         = {1385-7258},
	url          = {https://www.sciencedirect.com/science/article/pii/138572588490057X},
	mrnumber     = 748979,
	zbl          = {0589.32005},
	organization = {North-Holland}
}

@article{low-1985-embeddings,
	title        = {Embeddings and proper holomorphic maps of strictly pseudoconvex domains into polydiscs and balls},
	author       = {L{\o}w, Erik},
	year         = 1985,
	journal      = {Math. Z.},
	fjournal     = {Mathematische Zeitschrift},
	volume       = 190,
	number       = 3,
	pages        = {401--410},
	doi          = {10.1007/bf01215140},
	issn         = {0025-5874,1432-1823},
	url          = {https://doi.org/10.1007/BF01215140},
	mrclass      = {32h35},
	mrnumber     = 806898,
	mrreviewer   = {Yasuichiro\ Nishimura},
	zbl          = {0584.32048}
}

@article{faran-1986-linearity,
	title        = {The linearity of proper holomorphic maps between balls in the low codimension case},
	author       = {Faran, James John},
	year         = 1986,
	journal      = {J. Differential Geom.},
	fjournal     = {Journal of Differential Geometry},
	publisher    = {Lehigh University},
	volume       = 24,
	number       = 1,
	pages        = {15--17},
	doi          = {10.4310/jdg/1214440255},
	url          = {https://projecteuclid.org/journals/journal-of-differential-geometry/volume-24/issue-1/The-linearity-of-proper-holomorphic-maps-between-balls-in-the/10.4310/jdg/1214440255.full},
	mrnumber     = 857373,
	zbl          = {0592.32018}
}

@article{dangelo-1988-polynomial,
	title        = {Polynomial proper maps between balls},
	author       = {D'Angelo, John P.},
	year         = 1988,
	month        = aug,
	journal      = {Duke Math. J.},
	fjournal     = {Duke Mathematical Journal},
	publisher    = {Duke University Press},
	volume       = 57,
	number       = 1,
	pages        = {211--219},
	doi          = {10.1215/s0012-7094-88-05710-9},
	url          = {https://projecteuclid.org/journals/duke-mathematical-journal/volume-57/issue-1/Polynomial-proper-maps-between-balls/10.1215/S0012-7094-88-05710-9.short},
	mrnumber     = 952233,
	zbl          = {0657.32012}
}

@article{forstneric-1989-extending,
	title        = {Extending proper holomorphic mappings of positive codimension},
	author       = {Forstneri{\v{c}}, Franc},
	year         = 1989,
	month        = feb,
	journal      = {Invent. Math.},
	fjournal     = {Inventiones Mathematicae},
	publisher    = {Springer-Verlag Berlin/Heidelberg},
	volume       = 95,
	number       = 1,
	pages        = {31--61},
	doi          = {10.1007/bf01394144},
	url          = {https://link.springer.com/article/10.1007/BF01394144},
	mrnumber     = 969413,
	zbl          = {0633.32017}
}

@article{cima-1990-boundary,
	title        = {Boundary behavior of rational proper maps},
	author       = {Cima, Joseph A. and Suffridge, Teddy J.},
	year         = 1990,
	month        = feb,
	journal      = {Duke Math. J.},
	fjournal     = {Duke Mathematical Journal},
	publisher    = {Duke University Press},
	volume       = 60,
	number       = 1,
	pages        = {135--138},
	doi          = {10.1215/s0012-7094-90-06004-1},
	url          = {https://projecteuclid.org/journals/duke-mathematical-journal/volume-60/issue-1/Boundary-behavior-of-rational-proper-maps/10.1215/S0012-7094-90-06004-1.short},
	mrnumber     = 1047119,
	zbl          = {0694.32016}
}

@article{dor-1990-proper,
	title        = {Proper holomorphic maps between balls in one co-dimension},
	author       = {Dor, Avner},
	year         = 1990,
	journal      = {Ark. Mat.},
	fjournal     = {Arkiv f\"or Matematik},
	volume       = 28,
	number       = {1--2},
	pages        = {49--100},
	doi          = {10.1007/bf02387366},
	issn         = {0004-2080,1871-2487},
	url          = {https://projecteuclid.org/journals/arkiv-for-matematik/volume-28/issue-1-2/Proper-holomorphic-maps-between-balls-in-one-co-dimension/10.1007/BF02387366.full},
	mrclass      = {32h35 (32h40)},
	mrnumber     = 1049642,
	mrreviewer   = {T.\ J.\ Suffridge},
	zbl          = {0699.32014}
}

@book{dangelo-1993-several,
	title        = {Several complex variables and the geometry of real hypersurfaces},
	author       = {D'Angelo, John P.},
	year         = 1993,
	month        = jan,
	publisher    = {CRC Press},
	series       = {Studies in Advanced Mathematics},
	volume       = 8,
	isbn         = 9780849382727,
	edition      = 1,
	mrnumber     = 1224231,
	zbl          = {0854.32001}
}

@article{huang-1999-linearity,
	title        = {On a linearity problem for proper holomorphic maps between balls in complex spaces of different dimensions},
	author       = {Huang, Xiaojun},
	year         = 1999,
	journal      = {J. Differential Geom.},
	fjournal     = {Journal of Differential Geometry},
	publisher    = {Lehigh University},
	volume       = 51,
	number       = 1,
	pages        = {13--33},
	doi          = {10.4310/jdg/1214425024},
	url          = {https://projecteuclid.org/journals/journal-of-differential-geometry/volume-51/issue-1/On-a-linearity-problem-for-proper-holomorphic-maps-between-balls/10.4310/jdg/1214425024.full},
	mrnumber     = 1703603,
	zbl          = {1042.32008}
}

@article{dangelo-2003-sharp,
	title        = {A sharp bound for the degree of proper monomial mappings between balls},
	author       = {D'Angelo, John P. and Kos, {\v{S}}imon and Riehl, Emily},
	year         = 2003,
	month        = dec,
	journal      = {J. Geom. Anal.},
	fjournal     = {The Journal of Geometric Analysis},
	publisher    = {Springer},
	volume       = 13,
	number       = 4,
	pages        = {581--593},
	doi          = {10.1007/bf02921879},
	issn         = {1050-6926,1559-002x},
	url          = {https://link.springer.com/article/10.1007/BF02921879},
	mrnumber     = 2005154,
	zbl          = {1052.26016}
}

@article{dangelo-2003-Xvariety,
	title        = {Homogenization, reflection, and the {X}-variety},
	author       = {D'Angelo, John P.},
	year         = 2003,
	journal      = {Indiana Univ. Math. J.},
	fjournal     = {Indiana University Mathematics Journal},
	volume       = 52,
	number       = 5,
	pages        = {1113--1133},
	doi          = {10.1512/iumj.2003.52.2336},
	issn         = {0022-2518,1943-5258},
	url          = {https://doi.org/10.1512/iumj.2003.52.2336},
	mrclass      = {32h02 (32v20)},
	mrnumber     = 2010320,
	mrreviewer   = {Dmitri\ Zaitsev}
}

@article{dangelo-2003-proper,
	title        = {Proper holomorphic mappings, positivity conditions, and isometric imbedding},
	author       = {D'Angelo, John P.},
	year         = 2003,
	month        = may,
	journal      = {J. Korean Math. Soc.},
	fjournal     = {Journal of the Korean Mathematical Society},
	publisher    = {Korean Mathematical Society},
	volume       = 40,
	number       = 3,
	pages        = {341--371},
	doi          = {10.4134/jkms.2003.40.3.341},
	url          = {https://koreascience.kr/article/JAKO200315875840363.page},
	mrnumber     = 1973906,
	zbl          = {1044.32010}
}

@article{dangelo-2007-asian,
	title        = {The {$X$}-varieties for {CR} mappings between hyperquadrics},
	author       = {D'Angelo, John P.},
	year         = 2007,
	journal      = {Asian J. Math.},
	fjournal     = {Asian Journal of Mathematics},
	volume       = 11,
	number       = 1,
	pages        = {89--102},
	doi          = {10.4310/AJM.2007.v11.n1.a9},
	issn         = {1093-6106,1945-0036},
	url          = {https://doi.org/10.4310/AJM.2007.v11.n1.a9},
	mrclass      = {32v15 (32h35 32v40)},
	mrnumber     = 2304583,
	mrreviewer   = {Rasul\ Shafikov}
}

@article{dangelo-2009-complexity,
	title        = {On the complexity of proper holomorphic mappings between balls},
	author       = {D'Angelo, John P. and Lebl, Ji{\v{r}}{\'{i}}},
	year         = 2009,
	journal      = {Complex Var. Elliptic Equ.},
	fjournal     = {Complex Variables and Elliptic Equations. An International Journal},
	volume       = 54,
	number       = {3--4},
	pages        = {187--204},
	doi          = {10.1080/17476930902759403},
	issn         = {1747-6933,1747-6941},
	url          = {https://www.tandfonline.com/doi/abs/10.1080/17476930902759403},
	mrnumber     = 2513534,
	zbl          = {1171.32009}
}

@article{lebl-2011-normal,
	title        = {Normal forms, {Hermitian} operators, and {CR} maps of spheres and hyperquadrics},
	author       = {Lebl, Ji{\v{r}}{\'{i}}},
	year         = 2011,
	month        = nov,
	journal      = {Michigan Math. J.},
	fjournal     = {Michigan Mathematical Journal},
	publisher    = {University of Michigan, Department of Mathematics},
	volume       = 60,
	number       = 3,
	pages        = {603--628},
	doi          = {10.1307/mmj/1320763051},
	url          = {https://projecteuclid.org/journals/michigan-mathematical-journal/volume-60/issue-3/Normal-forms-hermitian-operators-and-CR-maps-of-spheres-and/10.1307/mmj/1320763051.full},
	mrnumber     = 2861091,
	zbl          = {1237.32005}
}

@article{dangelo-2012-pfisters,
	title        = {Pfister's theorem fails in the {Hermitian} case},
	author       = {D'Angelo, John P. and Lebl, Ji{\v{r}}{\'i}},
	year         = 2012,
	journal      = {Proc. Amer. Math. Soc.},
	fjournal     = {Proceedings of the American Mathematical Society},
	volume       = 140,
	number       = 4,
	pages        = {1151--1157},
	doi          = {10.1090/s0002-9939-2011-10841-4},
	issn         = {0002-9939,1088-6826},
	url          = {http://www.ams.org/jourcgi/jour-getitem?pii=S0002-9939-2011-10841-4},
	mrclass      = {12d15 (14p05 15b57 32v20)},
	mrnumber     = 2869101,
	mrreviewer   = {Murray\ Marshall},
	zbl          = {1309.12001}
}

@article{grundmeier-2014-rank,
	title        = {Bounding the rank of {H}ermitian forms and rigidity for {CR} mappings of hyperquadrics},
	author       = {Grundmeier, Dusty and Lebl, Ji\v{r}\'i and Vivas, Liz},
	year         = 2014,
	journal      = {Math.\ Ann.},
	fjournal     = {Mathematische Annalen},
	volume       = 358,
	number       = {3-4},
	pages        = {1059--1089},
	doi          = {10.1007/s00208-013-0989-z},
	issn         = {0025-5831,1432-1807},
	url          = {https://doi.org/10.1007/s00208-013-0989-z},
	mrclass      = {32v20},
	mrnumber     = 3175150,
	mrreviewer   = {Francine\ A.\ Meylan}
}

@book{dangelo-2019-hermitian,
	title        = {Hermitian analysis},
	author       = {D'Angelo, John P.},
	year         = 2019,
	month        = may,
	publisher    = {Springer},
	doi          = {10.1007/978-3-030-16514-7},
	isbn         = {978-3-030-16514-7},
	url          = {https://link.springer.com/book/10.1007/978-3-030-16514-7},
	edition      = 2,
	mrnumber     = 3931729,
	zbl          = {1428.42001},
	subtitle     = {From {Fourier} series to {Cauchy-Riemann} geometry}
}

@book{dangelo-2021-rational,
	title        = {Rational sphere maps},
	author       = {D'Angelo, John P.},
	year         = 2021,
	month        = jul,
	publisher    = {Springer},
	series       = {Progress in Mathematics},
	doi          = {10.1007/978-3-030-75809-7},
	isbn         = {978-3-030-75809-7},
	url          = {https://link.springer.com/book/10.1007/978-3-030-75809-7},
	edition      = 1,
	mrnumber     = 4293989,
	zbl          = {1485.32001}
}

@article{helal-2025-proper,
	title        = {Proper maps of ball complements \& differences and rational sphere maps},
	author       = {Helal, Abdullah Al and Lebl, Ji{\v{r}}{\'i} and Nandi, Achinta Kumar},
	year         = 2025,
	journal      = {Internat. J. Math.},
	fjournal     = {International Journal of Mathematics},
	volume       = 36,
	number       = 3,
	pages        = {Paper No. 2450079, 21},
	doi          = {10.1142/s0129167x24500794},
	issn         = {0129-167x,1793-6519},
	url          = {https://doi.org/10.1142/S0129167X24500794},
	mrclass      = {32h35 (32a08 32h02)},
	mrnumber     = 4859268,
	zbl          = {07978985}
}

\end{document}